\documentclass{amsart}
\usepackage{graphicx} 

\input{header}

\title{Abstract Independence Relations in Neostability Theory}
\author{Alberto Miguel-Gómez}
\date{\today}
\address{Department of Pure Mathematics and Mathematical Statistics, Centre for Mathematical Sciences, Wilberforce Road, Cambridge CB3 0WB, UK}
\email{am2962@cam.ac.uk}

\begin{document}

\maketitle
\begin{abstract}
    We develop a framework, in the style of Adler, for interpreting the notion of ``witnessing'' that has appeared (usually as a variant of Kim's Lemma) in different areas of neostability theory as a binary relation between abstract independence relations. This involves extending the relativisations of Kim-independence and Conant-independence due to Mutchnik to arbitrary independence relations. 
    
    After developing this framework, we show that several results from simplicity, $\NTP_2$, $\NSOP_1$, and beyond follow as instances of general theorems for abstract independence relations. In particular, we prove the equivalence between witnessing and symmetry and the implications from this notion to chain local character and the weak independence theorem, and recover some partial converses. Finally, we use this framework to prove a dichotomy between $\NSOP_1$ and Kruckman and Ramsey's $\BTP$ that applies to most known $\NSOP_4$ examples in the literature.
\end{abstract}
\tableofcontents
\section{Introduction}
The goal of this paper is twofold. First, we develop a framework that formalises the different notions of ``witnessing'' which have appeared across several papers in the neostability literature during the last few decades. These have traditionally received the name ``Kim's Lemma'', after the original version of this result in the context of simple theories (see \cite{kim1998forking}). In this paper, we treat witnessing, or Kim's Lemma, as a relation between two abstract independence relations. Equipped with this, we then show that many results from the theories of simplicity, $\NTP_2$, $\NSOP_1$ and beyond can be obtained as particular instances of general theorems for abstract independence relations. 

To explain this further, we give some general remarks about the previous development of abstract independence relations in model theory. The first explicit appearance of this notion in the model-theoretic literature is due to Makkai \cite{makkai1984survey}, and had the aim of reelaborating the results of Shelah on forking in stable theories from \cite[Chapter III]{shelah1978classification}. There are several results in the area that, viewed from our contemporary standpoint, can be restated as results about independence relations; an important example is Harnik and Harrington's characterisation of stable theories in terms of properties of certain extensions of types \cite{harnik1984fundamentals}. Treating non-forking as an independence relation also contributed to the development of geometric stability theory. 

The importance of independence relations came to the fore with Kim and Pillay's work on simple theories (see \cite{kim1998forking,kim1997simple,kim2001simplicity}). The key observation that properties of non-forking independence extend to and characterise the wider class of simple theories unified the work of several previous authors and brought powerful new tools for the development of other notions within simplicity. This also motivated the development of other classes of first-order theories characterised by the good behaviour of an independence relation, such as the class of rosy theories (see \cite{onshuus2006properties}).

As a generalisation of this work, Adler \cite{adler2009geometric} introduced the notion of an \textit{abstract independence relation} as a ternary relation on small subsets of a monster model of a fixed theory satisfying a number of properties. This treated independence relations as a ``semantic'' notion that abstracts from the more ``syntactic'' properties that characterise the traditional dividing lines, such as the order property in stable theories. It was also the first systematic attempt to ground many results from stable and simple theories on properties of these abstract independence relations. Moreover, \cite{adler2009geometric} shows that central results about simple theories (for instance, the fact that non-forking independence is symmetric) can be derived purely from properties of abstract independence relations. This also provided the framework in which the study of $\NTP_2$ theories took place (see \cite{chernikov2012forking,chernikov2014theories}). 

Recent developments in neostability have shown the limitations of Adler's original definition. The ternary relation characterising $\NSOP_1$ theories, known as \textit{Kim-independence} and developed by Kaplan, Shelah and Ramsey in \cite{kaplan2020kim,kaplan2019local,kaplan2021transitivity}, or the more recent notion of \textit{Conant-independence} from Mutchnik's work in \cite{mutchnik2024conant}, are generally far from satisfying all the properties that Adler required of an ``independence relation''. Moreover, there are several results across the different dividing lines (such as the equality of the relevant notions of forking and dividing) that point to an underlying unity not captured by Adler's framework. We show in this paper that many of these results can be fit into an abstract framework, similar to Adler's, which captures the notion of ``witnessing'' as a relation between independence relations.

This basic idea for the formalisation of the concept of witnessing, although original in its explicit formulation, is far from being so in its content. Many authors in recent years have encountered the issue of extending notions of witnessing beyond simple theories, and in doing so, they have provided the building blocks for the theory that is presented here. Thus, this work also aims to be a compilation of many results that neostability theory has established in recent years and a general reference to different areas of interest ($\NTP_2$, $\NSOP_1$, etc.).

A limitation of the present study is that we have restricted ourselves to working over models. It is an ongoing area of research how to extend the known results from this case in particular contexts (most notably $\NSOP_1$) to arbitrary bases (see, e.g., \cite{dobrowolski2022independence,mutchnik2024mathrm}). The project of extending the results contained in this paper to arbitrary bases over which $\ind$ satisfies full existence is left for future work. Let us note on this point that some results from \S\S\ref{sec:symmetry}-\ref{sec:independence-thm} have been independently obtained by Itay Kaplan and Nicholas Ramsey in unpublished work on defining Kim-independence over arbitrary bases in $\NSOP_1$ theories. The setting, purpose and extent of the two projects, however, are very different.

The structure of the paper is as follows. After a recap of traditional properties and constructions of independence relations in \S\ref{sec:independence}, we extend the notion of \textit{relative Kim-independence} introduced by Mutchnik in \cite{mutchnik2024conant} to a more general framework. This is done in two parts. In \S\ref{sec:strong-finite-character}, we introduce a new operation on abstract independence relations, in the style of Adler, that corresponds to ``forcing strong finite character''. With this operation in the background, we construct in \S\ref{sec:rel-kim-and-conant} the relativisation of Kim-independence to an independence relation $\ind$ from the more primitive relation of lifting the indiscernibility of Morley sequences (based on \cite{shelah1980simple}). This section also introduces a weak notion of relative Kim-dividing, which becomes relevant for the characterisations of witnessing in later sections. In this section, we also adapt the universal notion of Conant-independence appearing, beyond Mutchnik's work, in \cite{kaplan2019local}, \cite{kruckman2024new} and \cite{kim2022some}. 

In \S\ref{sec:witnessing} we introduce three different formulations of witnessing stemming from the literature, which we prove to be equivalent in most cases. We also include some basic results about witnessing here. The notion of witnessing we use for the rest of the paper is called GUWP (for \textit{generalised universal witnessing property}; see \thref{guwp-def}). Specifically, $\ind^1$ has GUWP w.r.t. $\ind^2$ if any formula that $\ind^2$-Kim-divides over a model $M \models T$ also $\ind^1$-Conant-divides over it. We develop the formal notion in this section and its equivalence to other renditions of ``Kim's lemma'' in the literature.

In \S\ref{sec:chain-local-character}, we move towards adapting more general results about the equivalence of witnessing to other traditional properties of abstract independence relations. In preparation for this, we start with \textit{chain local character} and show the following:
\begin{restateprop}{local-character-for-universal-kim-ind}
    Let $\ind$ be an independence relation that satisfies full existence, monotonicity, strong finite character, and right extension. Then $\ind$-Conant-independence satisfies chain local character. 
\end{restateprop}
As a corollary, we deduce that Conant-independence (in the sense of \cite{mutchnik2024conant}) satisfies chain local character in any theory. 

In \S\ref{sec:symmetry}, we recover the equivalence with symmetry:
\begin{restate}{symmetry-characterisation}
    Let $\ind$ be an independence relation satisfying full existence and monotonicity. Suppose that $\ind$-Kim-independence is compatible with $\ind$ and satisfies left transitivity. The following are equivalent:
    \begin{enumerate}[(i)]
        \item $\ind$-Kim-independence satisfies GUWP w.r.t. $\ind$. 
        \item $\ind$-Kim-independence is symmetric.
    \end{enumerate}
\end{restate}
We also show in this section how this allows us to recover one direction of \cite[Theorem 3.12]{mutchnik2024conant} from purely semantic arguments, as well as use the above result to relate witnessing to Adler's notion of an independence relation. 

In \S\ref{sec:independence-thm}, we consider relative versions of the Independence Theorem. We show that a general version of the ``weak independence theorem'' appearing in \cite{kaplan2020kim} and \cite{mutchnik2025nsop} follows from the assumption of witnessing and another property which we call ``weak transitivity'' (\thref{wit}). We also show that a slight strengthening of these axioms provides a new sufficient condition for $\NSOP_1$ theories (\thref{subclass-of-nsop1-characterisation}). We then prove the following partial converse to the weak independence theorem:
\begin{restate}{wit-and-guwp}
    Let $\ind^1 \implies \ind^2$ be independence relations satisfying full existence and monotonicity such that $\ind^1$ also satisfies left and right extension. If $\ind^2$-Conant-independence satisfies the $\ind^2$-independence theorem over models, then $\ind^1$ satisfies GUWP w.r.t. $\ind^2$. 
\end{restate}
This result allows us to show, in a purely semantic fashion, that if Conant-independence satisfies the independence theorem over models, then $T$ is $\NSOP_1$ (\thref{conant-independence-satisfying-wit-implies-nsop1}). This proof thus avoids the tree construction from \cite[Proposition 5.2]{chernikov2016model}.

Finally, in \S\ref{sec:dichotomies}, we prove a technical result that allows us to prove a dichotomy between $\NSOP_1$ and $\BTP$ for the class of theories with an independence relation $\ind$ satisfying full existence, monotonicity, and quasi-strong finite character, and such that $\iind$ satisfies GUWP w.r.t. $\ind$ (\thref{nsop1-or-btp-dichotomy}). In particular, this shows that most known strictly $\NSOP_4$ examples have $\BTP$. 
\subsection*{Acknowledgments}
The present work was completed during my PhD at Imperial College London and funded by the EPSRC (EP/W524323/1). The final revision was undertaken at the University of Cambridge under an LMS Early Career Fellowship. I am deeply grateful to my supervisors, David Evans and Charlotte Kestner, for their guidance, suggestions, and constant support. I would also like to thank Christian d'Elbée and Nicholas Ramsey for many useful comments and suggestions.
\subsection*{Keywords}
Model theory; neostability; independence relations; NSOP1; NSOP4
\section{Abstract independence relations}\label{sec:independence}
Let $T$ be a fixed first-order theory with a monster model $\M \models T$. 
\begin{definition}
    An \textbf{independence relation} $\ind$ is an $\Aut(\M)$-invariant ternary relation on subsets of $\M$. We write $A \ind_C B$ if $(A, B, C)$ lies in this relation, and say that $A$ is independent from $B$ over $C$. We say $\ind$ is an independence relation \textbf{over models} if $C$ is only allowed to be models of $T$.
\end{definition}
\begin{remark}
    Our definition of the informal notion of an independence relation is extremely general, and does not correspond to most formal uses of the notion previously appearing in the literature (e.g., \cite{adler2009geometric}). Its immediate precedent is found in what Chernikov and Kaplan call a ``preindependence relation'' in \cite{chernikov2012forking}. We favour this lax use of the term in this paper, since our purpose is to be explicit about the properties used in each result. 
\end{remark}  
\begin{convention}
    Let $a, b$ be tuples enumerating the sets $A, B$, respectively. We oftentimes write $a \ind_C b$ to mean $A \ind_C B$. 
\end{convention}
\begin{definition} \thlabel{properties-of-independence}
    The following is a list of common properties that independence relations may exhibit (where all sets and models are small):
    \begin{enumerate}[(i)]
        \item \textit{Left monotonicity}: if $A \ind_C B$ and $A' \subseteq A$, then $A' \ind_C B$.
        \item \textit{Right monotonicity}: if $A \ind_C B$ and $B' \subseteq B$, then $A \ind_C B'$. 
        \item \textit{Existence}: for any $A$ and $C$, $A \ind_C C$.
        \item \textit{Left extension}: if $A \ind_C B$ and $A \subseteq A'$, then there is some $B' \equiv_{CA} B$ such that $A' \ind_C B'$.
        \item \textit{Right extension}: if $A \ind_C B$ and $B \subseteq B'$, then there is some $A' \equiv_{CB} A$ such that $A' \ind_C B'$. 
        \item \textit{Left normality}: if $A \ind_C B$, then $CA \ind_C B$.
        \item \textit{Right normality}: if $A \ind_C B$, then $A \ind_C BC$.
        \item \textit{Left base monotonicity}: if $C \subseteq D \subseteq A$ and $A \ind_C B$, then $A \ind_D B$.
        \item \textit{Right base monotonicity}: if $C \subseteq D \subseteq B$ and $A \ind_C B$, then $A \ind_D B$.
        \item \textit{Left transitivity}: Given $D \subseteq C \subseteq B$, if $C \ind_D A$ and $B \ind_{C} A$, then $B \ind_D A$.
        \item \textit{Right transitivity}: Given $D \subseteq C \subseteq B$ if $A \ind_D C$ and $A \ind_C B$, then $A \ind_D B$.
        \item \textit{Strong finite character}: if $A \nind_C B$, then there are finite tuples $a \in A, b \in B$ and some $\phi(x,b) \in \tp(a/BC)$ such that, for all $a' \models \phi(x,b)$, $a' \nind_C b$.  
        \item \textit{Quasi-strong finite character}: for all types $p(x)$ and $q(y)$ over $C$, there is a set of $\mathcal{L}(C)$-formulas $\Sigma_{p,q}(x,y)$ such that 
        \begin{equation*}
            (a,b) \models \Sigma_{p,q}(x,y) \iff a \ind_C b \quad \& \quad \models p(a) \wedge q(b).
        \end{equation*}
        \item \textit{Local character}: For all $A$, there is a cardinal $\kappa = \kappa(A)$ such that, for all $B$, there is a set $C \subseteq B$ with $\size{C} < \kappa$ such that $A \ind_C B$.
        \item \textit{Symmetry}: if $A \ind_C B$, then $B \ind_C A$.
        \item \textit{Stationarity over models}: if $M \models T$, $A \ind_M B$, $A' \ind_M B$, and $A \equiv_M A'$, then $A \equiv_{MB} A'$.
        \item \textit{Independence Theorem over Models}: if $M \models T$, $A \equiv_M A'$, $A \ind_M B$, $A' \ind_M C$, and $B \ind_M C$, then there exists some $A''$ such that $A'' \equiv_{MB} A$, $A'' \equiv_{MC} A'$, and $A'' \ind_M BC$.  
    \end{enumerate}
\end{definition}
\begin{remark}
    \begin{enumerate}[(i)]
        \item We say $\ind$ has \textbf{monotonicity} if it has both left and right monotonicity. Similarly with all other ``left'' and ``right'' properties. 
        \item We say $\ind$ satisfies existence \textbf{over models} if, for any $A$ and $M \models T$, we have $A \ind_M M$. 
        \item Note that existence, right extension, and right monotonicity jointly imply the following property:
        \begin{enumerate}
            \item[$(*)$] \textit{Full existence}: For all $A, B, C$, there is $A' \equiv_C A$ such that $A' \ind_C B$. 
        \end{enumerate}
        Conversely, full existence implies existence.
    \end{enumerate}
\end{remark}
\begin{remark}
    Several of the above properties already appear, in some form, in \cite{shelah1978classification}. Their explicit versions in this abstract setting can be traced back to \cite{baldwin1988fundamentals}. The modern notation is due to \cite{adler2009geometric}. Some more recent properties appear in \cite{tent2013isometry,mutchnik2025nsop}. 
\end{remark}
\begin{example}[\protect{cf. \cite{casanovas2011simple,kaplan2020kim}}] \thlabel{ind-relations}
    Some common independence relations in the literature are the following:
    \begin{itemize}
        \item $A \aind_C B$ iff $\acl(AC) \cap \acl(BC) \subseteq \acl(C)$. In any theory, $\aind$ satisfies monotonicity, existence, left and right extension, normality, transitivity, strong finite character, local character, and symmetry.
        \item $A \iind_C B$ iff, for some (equiv., any) enumeration $a$ of $A$, $\tp(a/BC)$ extends to a global type which is Lascar-invariant over $C$. In any theory, $\iind$ satisfies monotonicity, existence over models, right extension, normality, right base monotonicity, left transitivity, and strong finite character. 
        \item $A \uind_C B$ iff, for some (equiv., any) enumeration $a$ of $A$, $\tp(a/BC)$ extends to a global type finitely satisfiable over $C$ (also known as a \textit{coheir extension}). In any theory, $\uind$ satisfies monotonicity, existence over models, left and right extension, normality, right base monotonicity, and strong finite character. 
        \item $A \find_C B$ iff, for some (equiv., any) enumeration $a$ of $A$, $\tp(a/BC)$ does not contain any formula that forks over $C$. In any theory, $\find$ satisfies monotonicity, existence over models, right extension, normality, right base monotonicity, left transitivity, and strong finite character. 
        \item For $M \models T$, $A \kind_M B$ iff, for some (equiv., any) enumeration $a$ of $A$, $\tp(a/MB)$ does not contain any formula that Kim-forks over $M$. Its properties will be given in a more general setting in \S\ref{sec:rel-kim-and-conant}.
    \end{itemize}
\end{example}
\begin{definition}
    Given two independence relations $\ind^1$ and $\ind^2$ on small subsets of $\M \models T$, we write $\ind^1 \implies \ind^2$ if, whenever $A \ind_C^1 B$, we have $A \ind_C^2 B$. 
\end{definition}
\begin{remark}
    The following are all the implications between the independence relations from \thref{ind-relations} in a general theory:
    \begin{equation*}
        {\ind}^\text{u} \implies {\ind}^{\text{i}} \implies {\ind}^{\text{f}} \implies {\ind}^\text{a}.
    \end{equation*}
    If we work over models, we also have $\find \implies \kind \implies \aind$.
\end{remark}
Adler observed in \cite{adler2009geometric} that several constructions in model theory can be obtained as the result of operations on abstract independence relations. We can find in the literature some common examples of unary and binary operations that allow us to build new independence relations out of old ones. 
\begin{definition}\thlabel{ind-star}
    Given an independence relation $\ind$, we define $\ind^*$ by
    \begin{equation*}
        A \: {\ind_C}^*\: B \iff \text{ for all } B' \supseteq B, \text{ there exists } A' \equiv_{BC} A \text{ such that } A' \ind_C B'.
    \end{equation*}
\end{definition}
\begin{fact}[\protect{\cite[Lemma 3.1]{adler2009geometric}}] \thlabel{ind-star-props} 
Let $\ind$ be an independence relation.
    \begin{enumerate}[(i)]
        \item $\ind^*$ is an independence relation and $\ind^* \implies \ind$.
        \item If $\ind^1 \implies \ind^2$ and $\ind^1$ satisfies right extension, then $\ind^1 \implies (\ind^2)^*$. In particular, if $\ind^1 \implies \ind^2$, then $(\ind^1)^* \implies (\ind^2)^*$.
        \item If $\ind$ satisfies monotonicity/right base monotonicity/left transitivity/left normality, then so does $\ind^*$.
        \item If $\ind$ satisfies monotonicity, then $\ind = \ind^*$ iff $\ind$ satisfies right extension. 
        \item If $\ind$ satisfies monotonicity and strong finite character, then $\ind^*$ satisfies strong finite character.
    \end{enumerate}
\end{fact}
\begin{example}
    As it is standard in the literature, let us write $A \dind_C B$ if, for some (equiv., any) enumeration $a$ of $A$, $\tp(a/BC)$ does not contain any formula that divides over $C$. Then $\ind^{\textnormal{f}} = (\ind^{\textnormal{d}})^*$ (see \cite[\S5]{adler2009geometric}).
\end{example}
The next operation has been explicitly defined by d'Elbée (\cite[Definition 3.2.9]{delbee2023axiomatic}):
\begin{definition}
    Given an independence relation $\ind$, we define $\ind^{\textnormal{opp}}$ by
    \begin{equation*}
        A \:{\ind_C}^{\text{opp}}\: B \iff B \ind_C A.
    \end{equation*}
\end{definition}
\begin{fact} \thlabel{properties-of-opp}
Let $\ind$ be an independence relation.
    \begin{enumerate}[(i)]
        \item $\ind^{\textnormal{opp}}$ is an independence relation.
        \item If $\ind^1 \implies \ind^2$ and $\ind^1$ satisfies symmetry, then $\ind^1 \implies (\ind^2)^{\textnormal{opp}}$. In particular, if $\ind^1 \implies \ind^2$, then $(\ind^1)^{\textnormal{opp}} \implies (\ind^2)^{\textnormal{opp}}$.
        \item If $\ind$ satisfies the left version of any axiom, then $\ind^{\textnormal{opp}}$ satisfies the right version. 
        \item $\ind = \ind^{\textnormal{opp}}$ iff $\ind$ satisfies symmetry. 
        \item If $\ind$ satisfies full existence, then so does $\ind^{\textnormal{opp}}$.
    \end{enumerate}
\end{fact}
\begin{example}\thlabel{hind}
    Again, using the standard notation and terminology from the literature, we write $A \hind_C B$ if, for some (equiv., any) enumeration $a$ of $A$, whenever $\phi(x,b) \in \tp(a/BC)$, there is some $c$ in $C$ such that $\phi(x,c) \in \tp(a/BC)$. It is folklore that $\ind^{\text{u}} = (\ind^{\text{h}})^{\text{opp}}$.
\end{example}
The third operation was introduced by the author in \cite[\S 5.2]{miguel2024classification}:
\begin{definition}
    Given an independence relation $\ind$, we define $\ind^{\text{le}}$ by 
    \begin{equation*}
        A \: {\ind_C}^{\textnormal{le}}\: B \iff \text{for all } D \supseteq A, \text{ there exists } B' \equiv_{AC} B \text{ with } D \ind_C B'.
    \end{equation*}
\end{definition}
\begin{fact}[\protect{\cite{miguel2024classification}}]
    $\ind^{\textnormal{le}} = \Big(\Big(\ind^{\textnormal{opp}}\Big)^*\Big)^{\textnormal{opp}}$.
\end{fact}
\begin{corollary}
Let $\ind$ be an independence relation.
    \begin{enumerate}[(i)]
        \item $\ind^\textnormal{le}$ is an independence relation and $\ind^{\textnormal{le}} \implies \ind$.
        \item If $\ind^1 \implies \ind^2$ and $\ind^1$ satisfies left extension, then $\ind^1 \implies (\ind^2)^{\textnormal{le}}$. In particular, if $\ind^1 \implies \ind^2$, then $(\ind^1)^{\textnormal{le}} \implies (\ind^2)^{\textnormal{le}}$.
        \item If $\ind$ satisfies monotonicity/left base monotonicity/right transitivity/right normality, then so does $\ind^{\textnormal{le}}$.
        \item If $\ind$ satisfies monotonicity, then $\ind = \ind^{\textnormal{le}}$ iff $\ind$ satisfies left extension.
    \end{enumerate}
\end{corollary}
We can also define higher-arity operations on several independence relations, such as binary operations that, out of two given independence relations, define a new one. Here, we develop some basic operations that capture most of the kinds of relations which have been defined in the literature on neostability theory. 
\begin{definition}
    Let $\ind^1$ and $\ind^2$ be two independence relations. We define $\ind^{1 \wedge 2}$ by:
\begin{equation*}
    A \:{\ind_C}^{1 \wedge 2}\: B \iff A \:{\ind_C}^1\: B \text{ and } A\: {\ind_C}^2\: B.
\end{equation*}
For notational convenience, we may also write $\ind^1 \wedge \ind^2$ instead of $\ind^{1\wedge 2}$.
\end{definition}
\begin{lemma} \thlabel{conjunction-props}
Let $\ind^1$ and $\ind^2$ be independence relations.
    \begin{enumerate}[(i)]
        \item $\ind^{1 \wedge 2}$ is an independence relation and $\ind^{1 \wedge 2} \implies \ind^1$ (as well as $\ind^{1 \wedge 2} \implies \ind^2$).
        \item If both $\ind^1$ and $\ind^2$ satisfy (strong) finite character/symmetry/normality/ base monotonicity/monotonicity/transitivity, then so does $\ind^{1 \wedge 2}$.
        \item $\ind^{1 \wedge 2}$ satisfies existence iff both $\ind^1$ and $\ind^2$ do. 
        \item If at least one of $\ind^1$ and $\ind^2$ satisfies stationarity over $C$, then so does $\ind^{1 \wedge 2}$.
        \item If $\ind^{1 \wedge 2}$ satisfies local character, then so do $\ind^1$ and $\ind^2$. Conversely, if $\ind^1$ and $\ind^2$ satisfy local character and right base monotonicity, then $\ind^{1\wedge 2}$ satisfies local character.
    \end{enumerate}
\end{lemma}
\begin{proof}
    It is immediate to check (i)-(iv) separately. We prove (v). The first part of the statement is clear. Suppose now that $\ind^1$ and $\ind^2$ both satisfy local character and right base monotonicity. Take $A$. By local character, there exist cardinals $\kappa_1,\kappa_2$ depending on $A$ such that, for any $B$, we can find $C_1,C_2 \subseteq B$ such that $\size{C_1} < \kappa_1$, $\size{C_2} < \kappa_2$, $A \ind_{C_1}^1 B$ and $A \ind_{C_2}^2 B$. Now, $\size{C_1 \cup C_2} < \kappa_1 + \kappa_2$ and $C_1\cup C_2 \subseteq B$. Thus, by right base monotonicity, we have $A \ind_{C_1\cup C_2}^1 B$ and $A \ind_{C_1 \cup C_2}^2 B$, and hence $A \ind_{C_1 \cup C_2}^{1\wedge 2} B$. This proves that $\kappa_1+\kappa_2$ witnesses local character for $\ind^{1\wedge 2}$.
\end{proof}
\begin{example}
    \begin{enumerate}[(i)]
        \item The independence relation define by the author in \cite{miguel2024classification} to prove $\TP_2$ is obtained via this operation: $\ind^{\text{hti}}$ is defined as $\ind^{\text{ht}} \wedge \iind$.
        \item We can extract from the definition of \textit{bi-invariant} types in \cite{hanson2023bi} a new independence relation $\ind^{\text{bi}}$ defined by $\iind \wedge (\iind)^{\text{opp}}$. 
    \end{enumerate}
\end{example}
We often make use of the following terminology:
\begin{definition}
    We say that independence relations $\ind^1$ and $\ind^2$ are \textbf{compatible} if $\ind^{1 \wedge 2}$ satisfies full existence. 
\end{definition}
\begin{example}
    \cite[Lemma 6.1.8]{miguel2024classification} shows that, in $T_{\mathbf{H}_4\text{-free}}$, $\ind^{\text{ht}}$ and $\iind$ are compatible. 
\end{example}
Inspired by the presentation of strict coheir independence from \cite{kim2022some}, we introduce the following abstract operation:
\begin{definition}
    Given two independence relations $\ind^1$ and $\ind^2$, we define \textbf{$\ind^2$-strict $\ind^1$} to be $(\ind^1 \wedge \: (\ind^2)^{\text{opp}})^*$. That is, $a$ is $\ind^2$-strict $\ind^1$ from $b$ over $D$ if, for all $c$, there is $c' \equiv_{Db} c$ such that $a \ind^1_D bc'$ and $bc' \ind_D^2 a$.
\end{definition}
\begin{lemma} \thlabel{strictness-props}
Let $\ind^1$ and $\ind^2$ be independence relations.
    \begin{enumerate}[(i)]
        \item $\ind^2$-strict $\ind^1$ is an independence relation and we have $\ind^2$-strict $\ind^1 \implies \ind^1$ and $\ind^2$-strict $\ind^1 \implies (\ind^2)^{\textnormal{opp}}$.
        \item If $\ind^1 = \ind^2$ and it satisfies symmetry, then $\ind^2$-strict $\ind^1 = (\ind^1)^*$.
        \item If $\ind^1$ has right (resp., left) monotonicity and $\ind^2$ has left (resp., right) monotonicity, then $\ind^2$-strict $\ind^1$ has right (resp., left) monotonicity. If both sides hold, then $\ind^2$-strict $\ind^1$ has right extension and right normality. 
        \item If $\ind^1$ has right base monotonicity and $\ind^2$ has left base monotonicity, then $\ind^2$-strict $\ind^1$ has right base monotonicity. 
        \item If $\ind^1$ has left transitivity and $\ind^2$ has right transitivity, then $\ind^2$-strict $\ind^1$ has left transitivity. 
        \item If $\ind^1 \implies \ind^3$ or $\ind^2 \implies \ind^4$, then $\ind^2$-strict $\ind^1 \implies \ind^4$-strict $\ind^3$.
    \end{enumerate}
\end{lemma}
\begin{proof}
    Combine all the different parts we have proven so far for each of the items in the above statement. 
\end{proof}
\begin{example} \thlabel{strict-relations}
    Several notions from the literature follow under this schema ($\kind$ is defined below):
    \begin{itemize}
        \item $\ind^{\text{st}}$ is defined by choosing both $\ind^1$ and $\ind^2$ to be $\find$.
        \item $\istind$ is defined by choosing $\ind^1$ to be $\iind$ and $\ind^2$ to be $\find$.
        \item $\kstind$ is defined by choosing $\ind^1$ to be $\iind$ and $\ind^2$ to be $\kind$. 
        \item $\ustind$ is defined by choosing $\ind^1$ to be $\uind$ and $\ind^2$ to be $\kind$.
    \end{itemize}
    All of the above relations satisfy monotonicity, right extension, and right normality. 

    Following the convention from the literature \cite{kruckman2024new}, we often say \textbf{Kim-strict $\ind^1$} instead of $\kind$-strict $\ind^1$.
\end{example}
Finally, let us introduce the following notion that originates with the work of Kim on simple theories \cite{kim1998forking} and is generalised by Adler in \cite{adler2009geometric}:
\begin{definition} \thlabel{morley-seqs-def}
    Let $\ind$ be an independence relation.
    \begin{enumerate}[(i)]
        \item We say $(a_i)_{i \in \omega}$ is an \textbf{$\ind$-independent sequence over $C$} if $a_i \ind_C a_{<i}$ for all $i \in \omega$.
        \item We say $(a_i)_{i \in \omega}$ is an \textbf{$\ind$-Morley sequence over $C$} if it is $C$-indiscernible and $\ind$-independent over $C$.
        \item Let $p$ be a global type. We say $(a_i)_{i \in \omega}$ is a \textbf{Morley sequence in $p$ over $C$} if $a_i \models p|_{Ca_{<i}}$ for all $i \in \omega$.
    \end{enumerate}
\end{definition}
The following folklore result will be used repeatedly without explicit mention:
\begin{fact}
    Suppose $\ind$ satisfies full existence and right monotonicity. For any tuple $a$ and set $C$, there exists an $\ind$-Morley sequence $(a_i)_{i < \omega}$ over $C$ with $a_0 = a$. 
\end{fact}
\section{Strong finite character} \label{sec:strong-finite-character}
In this section, we extend Adler's original programme by introducing a new operation to be understood as ``forcing strong finite character.'' As we will see, this allows us to formalise several intuitions about what occurs when forcing right extension, and provides the formal general framework for building relative Kim- and Conant-independence in the next section. 

Before doing this, let us introduce some terminology that will be useful. This expands the use of the term in \cite{chernikov2012forking} and unifies many properties of formulas which have appeared throughout neostability theory.
\begin{definition}\thlabel{our-freedom}
    We say a partial type $\pi(x)$ over $C$ is \textbf{$\ind$-free over $B$} if there is some $a \models \pi(x)$ such that $a \ind_B C$.
\end{definition}
\begin{lemma}
    Suppose $B \subseteq C$. A type $p(x) \in S(C)$ is $\ind$-free over $B$ iff, for all $a \models p$, we have $a \ind_B C$.
\end{lemma}
\begin{proof}
    ($\Rightarrow$) Suppose $p$ is $\ind$-free over $B$. Then there is some $a \models p$ such that $a \ind_B C$. Since $p$ is complete, for any $a' \models p$, $a' \equiv_{C} a$. By invariance, $a' \ind_B C$.

    ($\Leftarrow$) This follows immediately from the definition and the consistency of $p$.
\end{proof}
We will later see (\thref{sfc-and-free-types}) that, if $\ind$ satisfies strong finite character (\thref{properties-of-independence}(xii)), then the notion defined in \thref{our-freedom} coincides with what Chernikov and Kaplan call ``$\ind$-free'' in \cite{chernikov2012forking}.
\begin{example}
    The following characterisations are easy to check for a type $p(x) \in S(C)$ and a set $B \subseteq C$:
    \begin{enumerate}[(i)]
        \item $p(x)$ is $\uind$-free over $B$ iff it is finitely satisfiable in $B$. Moreover, a formula $\phi(x,c)$ is $\uind$-free over $B$ iff there is $b \in B$ such that $b \models \phi(x,c)$ (for the non-trivial direction, use the argument from, e.g., \cite[Lemma 8.1.3]{tent2012course}).
        \item $p(x)$ is $\iind$-free over $B$ iff it has a global extension Lascar-invariant over $B$. Moreover, it follows from \cite[Remark 2.19]{kruckman2024new} that a formula $\phi(x,c)$ is $\iind$-free over $B$ iff it does not quasi-fork over $B$.
        \item $p(x)$ is $\find$-free over $B$ iff it does not fork over $B$. Moreover, a formula $\phi(x,c)$ is $\find$-free over $B$ iff it does not fork over $B$ (see, e.g, \cite[Lemma 7.1.11]{tent2012course}).
    \end{enumerate}
\end{example}
Let us make a small digression to prove a new fact about $\aind$. For this, we need to introduce some terminology recalling and extending naturally \cite[Definition 2.1]{onshuus2006properties} (although the notion of strong dividing given in what follows is the updated version in \cite{hoffmann2023thorn}, which also appears in \cite{miguel2025stable}):
\begin{definition}
    Let $\phi(x,y)$ be a formula, $b$ a tuple, and $C$ a set of parameters. 
    \begin{enumerate}[(i)]
        \item We say $\phi(x,b)$ \textbf{strongly divides over $C$} if the set of images of $\phi(\M,b)$ under elements of $\Aut(\M/C)$ is infinite and $k$-inconsistent for some $k < \omega$.
        \item We say $\phi(x,b)$ \textbf{strongly forks over $C$} if there exist formulas $\psi_i(x,d_i)$ for $i < n$ such that $\phi(x,b) \vdash \bigvee_{i<n} \psi_i(x,d_i)$ and each $\psi_i(x,d_i)$ strongly divides over $C$.
    \end{enumerate}
\end{definition}
\begin{example}
    Let $T$ be the theory of infinite sets in the empty language. Pick two singletons $a,b$, and let $\phi(x,a,b)$ be the formula $x=a \vee x=b$. It is clear that each one of $x = a$ and $x = b$ strongly divides over $\varnothing$, and so by definition $\phi(x,a,b)$ strongly forks over $C$. Pick now a sequence of automorphisms $(\sigma_i)_{i \in \omega} \in \Aut(\M)$ such that $\sigma_i(ab)=ab_i$ where $b_i \neq b_j$ for all $i \neq j$. Then clearly $\bigcap_{i < \omega} \sigma_i(\phi(\M,a,b)) = \{a\}$, which shows that $\phi(x,a,b)$ does not strongly divide over $C$.
\end{example}
With this, we can provide a characterisation of $\aind$-free formulas, which is somewhat buried in Adler's proof of \cite[Proposition 4.7]{adler2008introduction}.
\begin{proposition} \thlabel{aind-free-and-strong-forking}
    Let $C$ be a set, $b$ be a tuple, and $\phi(x,y) \in \mathcal{L}(C)$. Then $\phi(x,b)$ is $\aind$-free over $C$ iff it does not strongly fork over $C$.
\end{proposition}
\begin{proof}
    ($\Rightarrow$) Let us first note that, by \cite[Proposition 4.7(3)$\Rightarrow$(4)]{adler2009geometric}, if a formula $\theta(x,c)$ strongly divides over $C$, then it is not $\aind$-free over $C$. Suppose now that $\phi(x,b)$ strongly forks over $C$. So there are formulas $\psi_i(x,d_i)$ for $i<n$ such that $\phi(x,b) \vdash \bigvee_{i<n} \psi_i(x,d_i)$ and each $\psi_i(x,d_i)$ strongly divides over $C$. Take $a \models \phi(x,b)$. By assumption, $a \models \psi_i(x,d_i)$ for some $i < n$, and by the above, this means $a \nind_C^{\text{a}} d_i$. So, by monotonicity, $a \nind_C^{\text{a}} bd_{<n}$. Since $a \models \phi(x,b)$ was arbitrary, it follows that, for any such $a$, we have $a \: (\nind_C^{\text{a}})^*\: b$, and thus $a \nind_C^{\text{a}} b$ by right extension. Therefore, $\phi(x,b)$ is not $\aind$-free over $C$.

    ($\Leftarrow$) We adapt the proof of \cite[Proposition 4.7(2)$\Rightarrow$(6)]{adler2008introduction}. Suppose $\phi(x,b)$ is not $\aind$-free over $C$. This means that, for every $d \models \phi(x,b)$, we have $d \nind_C^{\text{a}} b$.
    \begin{claim}
        There is a formula $\psi_d(x_d,b_d) \in \tp(d/Cb)$ that strongly divides over $C$.
    \end{claim}
    \begin{proof}[Proof of Claim]
        Since $d \nind^{\text{a}}_C b$, we can find some $e \notin \acl(C)$ such that $e \in \acl(Cd) \cap \acl(Cb)$. In particular, there is $\theta(x_d,y_d,z_d) \in \mathcal{L}$, a finite tuple $d' \subseteq d$, and a finite tuple $c \in C$ such that $\models \theta(e,d',c)$ and $\theta(x_d,d',c)$ is algebraic with $k$ many solutions. 

        Let $\psi_d(x_d,y_d,z_d) := \theta(x_d,y_d,z_d) \wedge \exists^{\leq k}w \: \theta(w,y_d,z_d)$. Then clearly $\models \psi_d(e,d',b)$, so that $\psi_d(e,y_d,c) \in \tp(d/Cb)$. Now observe that, since $e \notin \acl(C)$, the set of $C$-conjugates of $\psi_d(e,y_d,c)$ is infinite, and by definition, it is clearly $k$-inconsistent. Therefore, $\psi_d(e,y_d,c)$ strongly divides over $C$. Since by monotonicity we may assume $e \in Cb$, this completes the proof. \hfill \pushQED{$\qed_{\text{Claim}}$}
    \end{proof}
    Now, since $\{\phi(x,b)\} \cup \{\neg \psi_d(x_d,b_d) : d \models \phi(x,b)\}$ is inconsistent, there exist realisations $d_i$ of $\phi(x,b)$ for $i < n$ such that, after rewriting the variables,
    \begin{equation*}
        \phi(x,b) \vdash \bigvee_{i < n} \psi_{d_i}(x,b_{d_i}).
    \end{equation*}
    Therefore, $\phi(x,b)$ strongly forks over $C$.
\end{proof}
We often extend the terminology of $\ind$-freeness to global types as follows: for a small set $B$, we say a global type $p(x)$ is \textbf{$\ind$-free over $B$} if, for all small $B \subseteq C$ and $a \models p|_C$, we have $a \ind_B C$.
\begin{definition}
    Given an independence relation $\ind$, we define $\ind^{\text{sfc}}$ as follows:
    \begin{equation*}
        A \:{\ind_C}^{\text{sfc}}\: B \iff \text{for all } \phi(x,b) \in \tp(A/BC), \text{ there is } a' \models \phi(x,b) \text{ s.t. } a' \ind_C b. 
    \end{equation*}
\end{definition}
\begin{lemma} \thlabel{sfc-properties}
Let $\ind$ be an independence relation.
    \begin{enumerate}[(i)]
    \item $\ind^\sfc$ is an independence relation satisfying strong finite character and monotonicity. 
    \item If $\ind$ satisfies monotonicity, then $\ind \implies \ind^\sfc$.
    \item If $\ind^1 \implies \ind^2$ and $\ind^2$ satisfies strong finite character, then $(\ind^1)^\sfc \implies \ind^2$. In particular, if $\ind^1 \implies \ind^2$, then $(\ind^1)^{\sfc} \implies (\ind^2)^\sfc$. 
    \item If $\ind$ satisfies monotonicity, then $\ind = \ind^\sfc$ iff $\ind$ satisfies strong finite character. 
\end{enumerate}
\end{lemma}
\begin{proof}
    (i) Suppose $A \ind_C^{\text{sfc}} B$ and $\sigma \in \Aut(\M)$. Let $\phi(x,\sigma(b)) \in \tp(\sigma(A)/ \sigma(B) \sigma(C))$. Then $\phi(x,b) \in \tp(A/BC)$, and thus, by assumption, there is $a' \models \phi(x,b)$ such that $a' \ind_C b$. But then $\sigma(a') \models \phi(x,\sigma(b))$, and since $\ind$ is an independence relation, $\sigma(a') \ind_{\sigma(C)} \sigma(b)$. Therefore, $\sigma(A) \ind_{\sigma(C)}^{\text{sfc}} \sigma(B)$, as claimed.

    (Strong finite character) Suppose that $A \nind_C^{\text{sfc}} B$. By definition, there exists some $\phi(x,b) \in \tp(A/BC)$ such that, for all $a \models \phi(x,b)$, we have $a \nind_C b$. But notice that $\phi(x,b) \in \tp(a/Cb)$ for each such $a \models \phi(x,b)$. So, in fact, $a \nind_C^{\text{sfc}} b$.

    (Monotonicity) Suppose $A \ind_C^{\text{sfc}} B$ and $A' \subseteq A$, $B' \subseteq B$. Then $\tp(A'/B'C)$ is contained in the restriction of $\tp(A/BC)$ to the appropriate variables. In particular, if $\phi(x,b') \in \tp(A'/B'C)$, then $\phi(x,b') \in \tp(A/BC)$ and thus by assumption there is $a' \models \phi(x,b')$ such that $a' \ind_C b'$. Therefore, $A' \ind_C^{\text{sfc}} B'$.

    (ii) Suppose that $A \nind_C^{\text{sfc}} B$. So there exists a formula $\phi(x,b) \in \tp(A/BC)$ such that, for all $a' \models \phi(x,b)$, we have $a' \nind_C b$. Picking a tuple $a$ from $A$ such that $a \models \phi(x,b)$ and using monotonicity, it follows that $A \nind_C B$.

    (iii) Assume that $A \nind_C^2 B$. By strong finite character, there is a formula $\phi(x,b) \in \tp(A/BC)$ such that, for all $a' \models \phi(x,b)$, $a' \nind_C^2 b$. Thus, by assumption, $a' \nind_C^1 b$. Hence, by definition, $A \: (\nind_C^1)^{\text{sfc}} \: B$.

    (iv) From left to right, note first that, by (ii), we already have $\ind \implies \ind^{\sfc}$. Suppose now that $A \nind_C B$. By strong finite character, there exists some $\phi(x,b) \in \tp(A/BC)$ such that, for all $a' \models \phi(x,b)$, we have $a' \nind_C b$. But then by definition $A \nind_C^{\text{sfc}} B$. So $\ind = \ind^{\text{sfc}}$, as claimed. The converse follows from (i).
\end{proof}
The following formulation can be useful:
\begin{lemma} \thlabel{sfc-criterion}
    Let $\ind$ be an independence relation. Then $A \ind_C^{\textnormal{sfc}} B$ iff every formula $\phi(x,b) \in \tp(A/BC)$ is $\ind$-free over $C$.
\end{lemma}
\begin{proof}
    Simply note that $A \nind_C^{\text{sfc}} B$ iff there is a formula $\phi(x,b) \in \tp(A/BC)$ such that, for all $a \models \phi(x,b)$, $a \nind_C b$, iff there is $\phi(x,b) \in \tp(A/BC)$ which is not $\ind$-free over $C$.
\end{proof}
This allows us to generalise \cite[Remark 2.3]{hoffmann2023thorn}:
\begin{corollary} \thlabel{aind-and-strongly-dividing}
    In any theory, the following are equivalent: 
    \begin{enumerate}[(i)]
        \item $A \aind_C B$.
        \item For some (equiv., any) enumeration $a$ of $A$, $\tp(a/BC)$ does not contain a formula that strongly divides over $C$.
        \item For some (equiv., any) enumeration $a$ of $A$, $\tp(a/BC)$ does not contain a formula that strongly forks over $C$.
    \end{enumerate}
\end{corollary}
\begin{proof}
    First, observe that $\aind = (\aind)^\sfc$ by \thref{ind-relations}(i) and \thref{sfc-properties}(iv). Thus, by \thref{sfc-criterion}, $A \aind_C B$ iff every formula in $\tp(a/BC)$ (for $a$ an enumeration of $A$) is $\aind$-free over $C$. By \thref{aind-free-and-strong-forking}, this is equivalent to no formula in $\tp(a/BC)$ strongly forking over $C$. The equivalence with no formula strongly dividing over $C$ follows from the Claim in the proof of \thref{aind-free-and-strong-forking}.
\end{proof}
There are two immediate and useful applications of this characterisation for the theory of abstract independence relations. The first one formalises the fact that strong finite character only affects types:
\begin{lemma} \thlabel{free-formulas}
    Suppose $\ind$ satisfies monotonicity. A formula $\phi(x,b)$ is $\ind$-free over $C$ iff it is $(\ind)^\sfc$-free over $C$.
\end{lemma}
\begin{proof}
    ($\Rightarrow$) This follows directly from \thref{sfc-properties}(ii). 

    ($\Leftarrow$) Suppose $\phi(x,b)$ is $\ind^\sfc$-free over $C$. This means there is some $a \models \phi(x,b)$ such that $a \ind^\sfc_C b$. By \thref{sfc-criterion}, every formula in $\tp(a/Cb)$ is $\ind$-free over $C$. Therefore, $\phi(x,b)$ is $\ind$-free over $C$.
\end{proof}
The second one shows the equivalence of our notion of ``being $\ind$-free'' with that of Chernikov and Kaplan \cite{chernikov2012forking} under the assumption of strong finite character:
\begin{lemma}\thlabel{sfc-and-free-types}
    Suppose $\ind$ satisfies strong finite character. A type $p(x) \in S(C)$ is $\ind$-free over $B \subseteq C$ iff, for every $B \subseteq D \subseteq C$ and every $a \models p|_D$, we have $a \ind_B D$.
\end{lemma}
\begin{proof}
    Assume, for contradiction, that $p(x)$ is $\ind$-free over $B$ but there is some $B \subseteq D \subseteq C$ and $a \models p|_D$ such that $a \nind_B D$. By strong finite character, there is a formula $\phi(x,d) \in \tp(a/D)$ such that $\phi(x,d)$ is not $\ind$-free over $B$. But $\phi(x,d) \in p(x)$, and so by \thref{sfc-criterion} we obtain a contradiction. 
\end{proof}
Another useful result about the sfc operation comes from its interaction with the star operation from \thref{ind-star}, which implicitly appears in \cite[Lemma 1.2]{adler2009thorn}. The core of the argument can be traced back to a general characterisation of the result of forcing strong finite character and extension:
\begin{lemma} \thlabel{sfc-star-criterion}
    Let $\ind$ be an independence relation. Then $A \: (\ind^\sfc)^*_C \: B$ iff, for every formula $\phi(x,b) \in \tp(A/BC)$, if $\phi(x,b) \vdash \bigvee_{i < n} \psi_i(x,d_i)$, then there is some $i < n$ such that $\psi_i(x,d_i)$ is $\ind$-free over $C$.
\end{lemma}
\begin{proof}
    ($\Rightarrow$) Suppose there is some $\phi(x,b) \in \tp(A/BC)$ and formulas $\psi_i(x,d_i)$ for $i < n$ such that $\phi(x,b) \vdash \bigvee_{i < n} \psi_i(x,d_i)$ and no $\psi_i(x,d_i)$ is $\ind$-free over $C$. Pick $A' \equiv_{BC} A$. Then, if $a'$ is the corresponding subtuple of $A'$, we have $a' \models \phi(x,b)$, and so $a' \models \psi_i(x,d_i)$ for some $i < n$. So $\psi_i(x,d_i) \in \tp(A'/CBd_0\dots d_{n-1})$, and so by \thref{sfc-criterion} $A' \nind_C^\sfc Bd_0 \dots d_{n-1}$. Therefore, $A \: (\nind^\sfc)^*_C \: B$. 

    ($\Leftarrow$) Suppose that $A \: (\nind^\sfc)^*_C \: B$. This means there is some $D \supseteq B$ such that, for all $A' \equiv_{BC} A$, we have $A' \nind^\sfc_C D$. By \thref{sfc-criterion}, for each $A' \equiv_{BC} A$, we can find some $\psi_{A'}(x_{A'}, d_{A'}) \in \tp(A'/CD)$ that is not $\ind$-free over $C$.

    Now, the set $\tp(A/BC) \cup \{\neg \psi_{A'}(x_{A'},d_{A'}) : A' \models \tp(A/BC)\}$ is inconsistent. So, by compactness, we can find a formula $\phi(x,b) \in \tp(A/BC)$ and realisations $a_0,\dots, a_{n-1} \models \tp(A/BC)$ such that $\phi(x,b) \vdash \bigvee_{i < n} \psi_{a_i}(x_{a_i}, d_{a_i})$. Rewriting the variables in each formula, this gives $\phi(x,b) \vdash \bigvee_{i<n} \psi_{a_i}(x, d_{a_i})$, and each $\psi_{a_i}(x,d_{a_i})$ is not $\ind$-free over $C$. This proves the claim.
\end{proof}
\begin{lemma} \thlabel{commuting-star-and-sfc}
    Suppose that $\ind$ satisfies monotonicity. Then $(\ind^*)^{\textnormal{sfc}} = (\ind^{\textnormal{sfc}})^*$.
\end{lemma}
\begin{proof}
    ($\Rightarrow$) Suppose that $A \: (\ind^*)^\sfc_C \: B$. We will use \thref{sfc-star-criterion}. So assume that $\phi(x,b) \in \tp(A/BC)$ and there are formulae $\psi_i(x,d_i)$ for $i< n$ such that $\phi(x,b) \vdash \bigvee_{i < n} \psi_i(x,d_i)$. By assumption and \thref{sfc-criterion}, $\phi(x,b)$ is $\ind^*$-free over $C$. Thus, there is some $a' \models \phi(x,b)$ such that $a' \ind_C bd_0 \dots d_{n-1}$. But, by choice of formulae, we have $a' \models \psi_i(x,d_i)$ for some $i < n$. So, by monotonicity, $\psi_i(x,d_i)$ is $\ind$-free over $C$. Hence, by \thref{sfc-star-criterion}, $A \: (\ind^\sfc)^*_C \: B$.


    ($\Leftarrow$) Suppose that $A \: (\nind^*)^\sfc_C \: B$. So there is a formula $\phi(x,b) \in \tp(A/BC)$ such that, for all $a' \models \phi(x,b)$, we have $a' \nind_C^* b$. In particular, for $a$ the corresponding subtuple of $A$, we have $a \nind_C^* b$. So there is some $D \supseteq b$ such that, for all $a' \equiv_{Cb} a$, we have $a' \nind_C bD$, and so this holds in particular for all $a' \models \phi(x,b)$. Thus, for any such $a'$, $a' \nind_C^\sfc bD$. Hence, $a \: (\nind^\sfc)^*_C \: b$. Therefore, by monotonicity, $A \: (\nind^\sfc)^*_C \:B$, as claimed.
\end{proof}
With this, we can recover Adler's result (cf., \thref{ind-star-props}(v)) as a special case:
\begin{proposition}[\protect{\cite[Lemma 1.2]{adler2009thorn}}]
    Suppose $\ind$ satisfies monotonicity. If $\ind = \ind^{\textnormal{sfc}}$, then $\ind^* = (\ind^*)^{\textnormal{sfc}}$. 
\end{proposition}
Our result also extends Adler's result in the following way:
\begin{proposition}\thlabel{sfc-preserves-right-ext}
    Suppose $\ind$ satisfies monotonicity. If $\ind = \ind^*$, then $\ind^{\textnormal{sfc}} = (\ind^{\textnormal{sfc}})^*$.

    Equivalently, if $\ind$ satisfies monotonicity and right extension, then $\ind^{\textnormal{sfc}}$ also satisfies right extension.
\end{proposition}
We can also provide a concrete description of how adding strong finite character and extension works on formulas:
\begin{proposition} \thlabel{sfc-star-free-formula}
    Suppose $\ind$ satisfies monotonicity. A formula $\phi(x,b)$ is $(\ind^\sfc)^*$-free over $C$ iff there are no formulae $\psi_i(x,d_i)$ for $i < n$ such that $\phi(x,b) \vdash \bigvee_{i < n} \psi_i(x,d_i)$ and no $\psi_i(x,d_i)$ is $\ind$-free over $C$.
\end{proposition}
\begin{proof}
    ($\Rightarrow$) Suppose $\phi(x,b)$ is $(\ind^\sfc)^*$-free over $C$. Let $a \models \phi(x,b)$ be such that $a \: (\ind^\sfc)^*_C \: b$. By \thref{sfc-star-criterion}, as $\phi(x,b) \in \tp(a/Cb)$, we conclude that, whenever $\phi(x,b) \vdash \bigvee_{i<n} \psi_i(x,d_i)$, there is some $i < n$ such that $\psi_i(x,d_i)$ is $\ind$-free over $C$. 

    ($\Leftarrow$) We adapt the classical proof for non-forking (cf., \cite[Lemma 7.1.11]{tent2012course}). Suppose that the right-hand side holds. Let $\pi(x,b)$ be a maximal partial type over $Cb$ containing $\phi(x,b)$ such that every formula in $\pi$ satisfies the right-hand side. 
    
    We claim that $\pi(x,b)$ is complete. Indeed, assume, for contradiction, that it is not. Then there is a formula $\psi(x,b') \in \mathcal{L}(Cb)$ such that $\psi(x,b'), \neg \psi(x,b') \notin \pi(x,b)$. By definition of $\pi$, this means that there exist formulae $\theta_i(x,d_i)$ for $i < n$ and $\rho_j(x,e_j)$ for $j < m$ such that $\psi(x,b') \vdash \bigvee_{i<n} \theta_i(x,d_i)$, $\neg \psi(x,b') \vdash \bigvee_{j < m} \rho_j(x,e_j)$, and no $\theta_i$, $\rho_j$ is $\ind$-free over $C$. Now, $\pi(x,b) \vdash \psi(x,b') \vee \neg \psi(x,b')$. By compactness, there is some $\eta(x,b'') \in \pi(x,b)$ such that $\eta(x,b'') \vdash \psi(x,b') \vee \neg \psi(x,b')$. But then $\eta(x,b'') \vdash \bigvee_{i <n} \theta_i(x,d_i) \vee \bigvee_{j < m} \rho_j(x,d_j)$, contradicting our definition of $\pi$.

    Since $\pi(x,b)$ is complete, by \thref{sfc-star-criterion}, for any $a \models \pi(x,b)$ we have $a \: (\ind^\sfc)^*_C \: b$. Since $\phi(x,b) \in \pi(x,b)$, this means that $\phi(x,b)$ is $(\ind^\sfc)^*$-free over $C$, as claimed.
\end{proof}
\begin{corollary}
    Suppose $\ind$ satisfies monotonicity. A formula $\phi(x,b)$ is $\ind^*$-free over $C$ iff there are no formulae $\psi_i(x,d_i)$ for $i<n$ such that $\phi(x,b) \vdash \bigvee_{i < n} \psi_i(x,d_i)$ and no $\psi_i(x,d_i)$ is $\ind$-free over $C$.
\end{corollary}
\begin{proof}
    By \thref{sfc-star-free-formula}, it suffices to show that a formula is $\ind^*$-free over $C$ iff it is $(\ind^*)^\sfc$-free over $C$. But this follows from \thref{free-formulas}.
\end{proof}
\section{Relative Kim- and Conant-independence} \label{sec:rel-kim-and-conant}
We now move on to the main objects of study in this paper: the relativised versions of Kim- and Conant-independence. Both originate in the work of Mutchnik for more restricted cases: in \cite{mutchnik2024conant}, $\ind$-Kim-independence is defined whenever $\ind$ is an independence relation satisfying full existence, monotonicity, and stationarity over models, and \cite{mutchnik2025nsop,mutchnik2024conant} introduce $\ind$-Conant-independence for $\ind = \iind,\uind$. In this section, we want to generalise these notions to any choice of independence relation. 

However, our presentation will be different to what is usual in the current literature on Kim-independence (cf., \cite{kaplan2020kim,mutchnik2024conant}). Instead of defining, e.g., $\ind$-Kim-independence as a primitive notion and then noting that the ``basic characterisation'' of $\ind$-Kim-dividing holds (cf., \cite[Lemma 3.18]{kaplan2020kim}), we will start with the lifting of indiscernibility for $\ind$-Morley sequences as an independence relation and build out of this $\ind$-Kim-independence. This procedure is structurally analogous to Adler's abstract definition of $\dind$ and $\find$ and his clarification of their relationship in \cite{adler2009geometric}.

For the benefit of the reader, let us note that this level of detail is presented for conceptual clarification; for the rest of this paper, it is enough for the reader to take \thref{properties-of-kim-independence,conant-ind-defined-right} as definitions instead of results.
\begin{remark}
    All the independence relations we will introduce from this point onwards will be defined only over models. This requires a suitable modification of the properties listed in \thref{properties-of-independence}. This is standard, but we explicitly describe it. For monotonicity, extension, normality, strong finite character, quasi-strong finite character, and symmetry, the version over models is clear. For (left) transitivity, we need to choose models $M \prec N \models T$, and we recast the property as ``if $N \ind_M A$ and $B \ind_N A$, then $NB \ind_M A$''. For (left) base monotonicity, we again need to take models $M \prec N$ which are contained in a set $A$, and redefine this as ``if $A \ind_M B$, then $A \ind_N B$''. The case of local character is more subtle; we will introduce in \S\ref{sec:chain-local-character} the version of local character over models that we will use.
\end{remark}
\subsection{Relative Kim-independence}
We begin by introducing the following independence relation, inspired by \cite[Lemma 1.4]{shelah1980simple} (observe the similarity with \cite[\S5]{adler2008introduction}):
\begin{definition}
    Let $\ind$ be an independence relation. We define $\ind^+$ as follows: $A \ind_M^+ B$ iff, for every $\ind$-Morley sequence $I$ over $M$ starting at $B$, there is some $I' \equiv_{MB} I$ which is $MA$-indiscernible.
\end{definition}
Observe that this definition implicitly assumes the (easy) fact that the choice of enumeration for $B$ does not matter. From now on, we will omit explicit mention of underlying enumerations whenever there is no confusion.
\begin{lemma} \thlabel{properties-of-ind-plus}
    Let $\ind$ be an independence relation. Then $\ind^+$ is an independence relation satisfying existence and left normality.
\end{lemma}
\begin{proof}
    We first check that $\ind^+$ is an independence relation. Suppose that $A \ind_M^+ B$ and $\sigma \in \Aut(\M)$. Let $(B'_i)_{i < \omega}$ be an $\ind$-Morley sequence over $\sigma(M)$ with $B'_0 = \sigma(B)$. Then $(\sigma^{-1}(B'_i))_{i < \omega}$ is $\ind$-Morley over $M$ and $\sigma^{-1}(B'_0) = B$. Since $A \ind_M^+ B$, we can find some $A' \equiv_{MB} A$ such that $(\sigma^{-1}(B'_i))_{i<\omega}$ is $MA'$-indiscernible. Hence, $\sigma(A') \equiv_{\sigma(M)\sigma(B)} \sigma(A)$ and $(B'_i)_{i < \omega}$ is $\sigma(M)\sigma(A')$-indiscernible. Therefore, we get $\sigma(A) \ind_{\sigma(M)}^+ \sigma(B)$.

    (Existence) Suppose $(M_i)_{i < \omega}$ is $\ind$-Morley over $M$ and $M_0 = M$. By indiscernibility, it follows that $M_i = M$ for all $i < \omega$. So clearly, for any $A$, the sequence $(M_i)_{i < \omega}$ is $MA$-indiscernible. Therefore, $A \ind_M^+ M$.

    (Left normality) This is immediate. 
\end{proof}
Of course, if $\ind$-Morley sequences do not exist, then the above independence relation is trivial. We often add the assumptions of full existence and monotonicity to ensure non-triviality.
\begin{lemma} \thlabel{dividing-for-partial-types}
    If $\pi(x,b)$ is $\ind^+$-free over $M$, then, for every $\ind$-Morley sequence $(b_i)_{i < \omega}$ over $M$ with $b_0 = b$, the set $\bigcup_{i < \omega} \pi(x,b_i)$ is consistent. 
\end{lemma}
\begin{proof}
    Suppose $\pi(x,b)$ is $\ind^+$-free over $M$. Let $I = (b_i)_{i < \omega}$ be $\ind$-Morley over $M$ with $b_0 = b$. Since $\pi(x,b)$ is $\ind^+$-free, we can find some $c \models \pi(x,b)$ such that $c \ind_M^+ b$. By an automorphism, this gives some $c' \equiv_{Mb} c$ such that $I$ is $Mc'$-indiscernible. But $c' \models \pi(x,b)$, and so $c' \models \bigcup_{i < \omega} \pi(x,b_i)$ by indiscernibility. 
\end{proof}
\begin{lemma} \thlabel{dividing-for-complete-types}
    A type $p(x,b) \in S(Mb)$ is $\ind^+$-free over $M$ iff, for every $\ind$-Morley sequence $(b_i)_{i < \omega}$ over $M$ with $b_0 = b$, the set $\bigcup_{i < \omega} p(x,b_i)$ is consistent.
\end{lemma}
\begin{proof}
    ($\Rightarrow$) This is \thref{dividing-for-partial-types}.

    ($\Leftarrow$) Let $I = (b_i)_{i < \omega}$ be a $\ind$-Morley sequence over $M$ with $b_0 = b$. By assumption, $\bigcup_{i < \omega} p(x,b_i)$ is consistent. Let $a' \models \bigcup_{i < \omega} p(x,b_i)$. By Ramsey, compactness, an automorphism and the completeness of $p$, we may assume $I$ is $Ma'$-indiscernible. Moreover, by assumption, $a' \models p(x,b)$. So $p$ is $\ind^+$-free over $M$.
\end{proof}
\begin{proposition} \thlabel{sfc-for-ind-plus}
    If $\ind$ satisfies monotonicity, then $\ind^+$ satisfies strong finite character.
\end{proposition}
\begin{proof}
    We use \thref{sfc-criterion}. Suppose that $A \nind_M^+ B$. Fix an enumeration $a$ of $A$ and $b$ of $B$ and let $p(x,y) := \tp(ab/M)$. Then $p(x,b)$ is not $\ind^+$-free over $M$. By \thref{dividing-for-complete-types}, there is some $\ind$-Morley sequence $(b_i)_{i < \omega}$ over $M$ with $b_0 = b$ such that $\bigcup_{i < \omega} p(x,b_i)$ is inconsistent. By compactness, there is a formula $\phi(x,b') \in p(x,b)$ with $b' \subseteq b$ finite such that $\bigcup_{i < \omega} \{\phi(x,b'_i)\}$ is inconsistent, where $b'_i$ denotes the restriction of $b_i$ to the appropriate variables. Therefore, by monotonicity and \thref{dividing-for-partial-types}, $\phi(x,b')$ is not $\ind^+$-free over $M$. Therefore, $A \: (\nind^+)^{\text{sfc}}_M \: B$.
\end{proof}
\begin{lemma}\thlabel{ind-kim-fork}
    Suppose that $\ind$ satisfies monotonicity. A formula $\phi(x,b)$ is $(\ind^+)^*$-free over $M$ iff, for every $\psi_i(x,d_i)$ for $i < n$ such that $\phi(x,b) \vdash \bigvee_{i < n} \psi_i(x,d_i)$, there is some $i < n$ such that, for every $\ind$-Morley sequence $(d_{i,j})_{j<\omega}$ over $M$ with $d_{i,0}  = d_i$, the set $\{\psi_i(x,d_{i,j}) : j < \omega\}$ is consistent. 
\end{lemma}
\begin{proof}
    ($\Rightarrow$) By \thref{sfc-for-ind-plus,sfc-star-free-formula}, we have that $\phi(x,b)$ is $(\ind^+)^*$-free over $M$ iff, for every $\psi_i(x,d_i)$ such that $\phi(x,b) \vdash \bigvee_{i < n} \psi_i(x,d_i)$, there is $i < n$ such that $\psi_i(x,d_i)$ is $\ind^+$-free over $C$. By \thref{dividing-for-partial-types}, it follows that, for every $\ind$-Morley sequence $(d_{i,j})_{j < \omega}$ over $M$ with $d_{i,0} = d_i$, the set $\{\psi_i(x,d_{i,j}) : j < \omega\}$ is consistent.  

    ($\Leftarrow$) Suppose that $\phi(x,b)$ is not $(\ind^+)^*$-free over $M$. Let $a \models \phi(x,b)$. Then $a \: (\nind_M^+)^* \: b$, so there is some $d$ such that $a \nind_M^+ bd$. Let $p_a(x,b,d) := \tp(a/Mbd)$. By \thref{dividing-for-complete-types}, there is an $\ind$-Morley sequence $(b_id_i)_{i < \omega}$ over $M$ with $b_0 = b$ such that $\bigcup_{i < \omega} p_a(x,b_i,d_i)$ is inconsistent. By compactness, we can find some formula $\psi_a(x,b_a,d_a) \in p_a(x,b,d)$ with $b_a \subseteq b$, $d_a \subseteq d$ such that $\{\psi_a(x,b_{a,i},d_{a,i}) :i < \omega\}$ is inconsistent, where $b_{a,i}d_{a,i}$ denotes the restriction of $b_id_i$ to the appropriate variables. 

    Now, the set $\{\phi(x,b)\} \cup \{\neg \psi_a(x,b_a,d_a) : a \models \phi(x,b)\}$ is inconsistent. By compactness, we can thus find realisations $a_0, \dots, a_{n-1} \models \phi(x,b)$ such that 
    \begin{equation*}
        \phi(x,b) \vdash \bigvee_{i < n} \psi_{a_i}(x,b_{a_i},d_{a_i}).
    \end{equation*}
    This proves the claim.
\end{proof}
\begin{definition}
Let $\ind$ be an independence relation. 
\begin{enumerate}[(i)]
    \item We define \textbf{$\ind$-Kim-independence} to be $(\ind^+)^*$.
    \item We say $\phi(x,b)$ \textbf{$\ind$-Kim-divides over $M$} if there is an $\ind$-Morley sequence $(b_i)_{i < \omega}$ over $M$ with $b_0 = b$ such that $\{\phi(x,b_i) : i < \omega\}$ is inconsistent.
    \item We say $\phi(x,b)$ \textbf{weakly $\ind$-Kim-divides over $M$} if it is not $\ind^+$-free over $M$.
    \item We say $\phi(x,b)$ \textbf{$\ind$-Kim-forks over $M$} if it is not $(\ind^+)^*$-free over $M$.
\end{enumerate}
\end{definition}
\begin{lemma} \thlabel{properties-of-kim-independence}
    Let $\ind$ be an independence relation satisfying monotonicity.
    \begin{enumerate}[(i)]
        \item If $\phi(x,b)$ $\ind$-Kim-divides over $M$, then it weakly $\ind$-Kim-divides over $M$.
        \item A formula $\phi(x,b)$ $\ind$-Kim-forks over $M$ iff there exist formulae $\psi_i(x,d_i)$ for $i < n$ such that $\phi(x,b) \vdash \bigvee_{i < n} \psi_i(x,d_i)$ and each $\psi_i(x,d_i)$ $\ind$-Kim-divides over $M$.
        \item We have $A$ is $\ind$-Kim-independent from $B$ over $M$ iff, for some (equiv., any) enumeration $a$ of $A$, $\tp(a/MB)$ does not contain any formula that $\ind$-Kim-forks over $M$.
    \end{enumerate}
\end{lemma}
\begin{proof}
    (i) This follows directly from \thref{dividing-for-partial-types}.

    (ii) This is a rephrasing of \thref{ind-kim-fork}.

    (iii) This follows from strong finite character for $(\ind^+)^*$ by \thref{sfc-criterion}.
\end{proof}
\begin{example} \thlabel{kim-independence}
    We refer to $\iind$-Kim-independence as \textbf{Kim-independence}, and denote it by $\kind$. We will later show that this notion coincides with that first introduced in \cite{kaplan2020kim} and which we already defined in \thref{ind-relations}.
\end{example}
\begin{remark}
    We extend the terminology above for complete types as follows: we say $p(x) \in S(Mb)$ \textbf{$\ind$-Kim-divides over $M$} (resp., \textbf{$\ind$-Kim-forks over $M$}) if it is not $\ind^+$-free (resp., $(\ind^+)^*$-free) over $M$. Using \thref{dividing-for-complete-types} and monotonicity of $\ind$, it is easy to show that $p(x)$ $\ind$-Kim-divides over $M$ iff there is a formula $\phi(x,b') \in p(x)$ such that $\phi(x,b')$ $\ind$-Kim-divides over $M$. 
\end{remark}
The notion of ``weak $\ind$-Kim-dividing'' is new and arises from the attempt to build $\ind$-Kim-independence directly from the condition of lifting indiscernibility. It would be very satisfactory if the notions of $\ind$-Kim-dividing and weak $\ind$-Kim-dividing coincided. However, this only happens in the limit case when $\ind$-Kim-dividing coincides with $\ind$-Kim-forking:
\begin{proposition}\thlabel{kim-dividing-and-freeness}
    Let $\ind$ be an independence relation satisfying monotonicity. The following are equivalent:
    \begin{enumerate}[(i)]
        \item If $\pi(x,b)$ weakly $\ind$-Kim-divides over $M$, then it $\ind$-Kim-divides over $M$. 
        \item There is an independence relation $\ind^\oplus$ such that $\pi(x,b)$ does not $\ind$-Kim-divide over $M$ iff it is $\ind^\oplus$-free over $M$. 
        \item If $\pi(x,b)$ does not $\ind$-Kim-divide over $M$, then there is a complete type $p(x) \in S(Mb)$ containing $\pi(x,b)$ such that $p(x)$ \textbf{does not $\ind$-Kim-divide over $M$}, i.e., for every $\ind$-Morley sequence $(b_i)_{i < \omega}$ over $M$ with $b_0 = b$, the set $\bigcup_{i < \omega} p(x,b_i)$ is consistent.
        \item If $\pi(x,b)$ $\ind$-Kim-forks over $M$, then it $\ind$-Kim-divides over $M$.
    \end{enumerate}
\end{proposition}
\begin{proof}
%
    ((i) $\Rightarrow$ (ii)) Clear by (i) and \thref{dividing-for-partial-types}.

    ((ii) $\Rightarrow$ (iii)) Suppose such an independence relation exists. 
    \begin{claim}
        $(\ind^\oplus)^\sfc = \ind^+$.
    \end{claim}
    \begin{proof}[Proof of Claim]
        ($\Rightarrow$) Suppose that $A \nind_M^+ B$. Let $p(x,b) := \tp(A/MB)$ (where $b$ enumerates $B$). By \thref{dividing-for-complete-types}, there is an $\ind$-Morley sequence $(b_i)_{i < \omega}$ over $M$ with $b_0 = b$ such that $\bigcup_{i<\omega} p(x,b_i)$ is inconsistent. By compactness, there is a formula $\phi(x,b') \in p(x,b)$ with $b' \subseteq b$ finite such that $\{\phi(x,b'_i) : i < \omega\}$ is inconsistent, i.e., $\ind$-Kim-divides over $M$. By our choice of independence relation, $\phi(x,b')$ is not $\ind^\oplus$-free over $M$. Hence $A \: (\nind^\oplus)^\sfc_M \: B$. 

        ($\Leftarrow$) Suppose $A \: (\nind^\oplus)^\sfc_M \: B$. By \thref{sfc-criterion}, there is some $\phi(x,b) \in \tp(A/MB)$ which is not $\ind^\oplus$-free over $M$. By definition, $\phi(x,b)$ $\ind$-Kim-divides over $M$. Hence, by \thref{dividing-for-partial-types}, $\phi(x,b)$ is not $\ind^+$-free over $M$. Therefore, by \thref{sfc-criterion,sfc-for-ind-plus}, $A \nind^+_M B$. \hfill \pushQED{$\qed_{\text{Claim}}$}
    \end{proof}
    Therefore, by \thref{free-formulas}, if $\pi(x,b)$ does not $\ind$-Kim-divide over $M$, then it is $\ind^\oplus$-free over $M$, hence also $(\ind^\oplus)^\sfc$-free over $M$. Thus, there is $a \models \pi(x,b)$ such that $a \: (\ind^\oplus)^\sfc_M \: b$. By the Claim, $a \ind_M^+ b$. By \thref{dividing-for-complete-types}, $\tp(a/Mb)$ is the type we wanted.
    
    ((iii) $\Rightarrow$ (iv)) Assume, for contradiction, that $\pi(x,b)$ $\ind$-Kim-forks but does not $\ind$-Kim-divide over $M$. By \thref{properties-of-kim-independence}(ii), there are formulae $\psi_i(x,d_i)$ for $i < n$ such that $\pi(x,b) \vdash \bigvee_{i < n} \psi_i(x,d_i)$ and each $\psi_i(x,d_i)$ $\ind$-Kim-divides over $M$. 
    
    Now, observe there is a partial type $\rho(x,b,d_{<n})$ such that $\pi(x,b) \leftrightarrow \rho(x,b,d_{<n})$ holds. Let $(b_id_{0,i} \dots d_{n-1,i})_{i < \omega}$ be an $\ind$-Morley sequence over $M$ starting at $bd_{<n}$. By monotonicity, $(b_i)_{i < \omega}$ is an $\ind$-Morley sequence over $M$ starting at $b$, and thus, by assumption, $\bigcup_{i < \omega} \pi(x,b_i)$ is consistent. Since $b_id_{<n,i} \equiv_M bd_{<n}$, it follows that $\pi(x,b_i) \leftrightarrow \rho(x,b_i,d_{<n,i})$ for each $i$. Therefore, $\bigcup_{i < \omega} \rho(x,b_i,d_{<n,i})$ is consistent. 

    By (iii), there is a complete type $q(x) \in S(Mbd_0 \dots d_{n-1})$ such that $\rho(x,b,d_{<n}) \subseteq q(x)$ and $q$ does not $\ind$-Kim-divide over $M$. By our choice of formulae, there is some $i < n$ such that $\psi_i(x,d_i) \in q(x)$. But then $q(x)$ does $\ind$-Kim-divide over $M$, a contradiction. 

    ((iv) $\Rightarrow$ (i)) Suppose $\pi(x,b)$ does not $\ind$-Kim-divide over $M$. By (iv), it does not $\ind$-Kim-fork over $M$, and so by definition it is $(\ind^+)^*$-free over $M$. In particular, $\pi(x,b)$ is $\ind^+$-free over $M$.
\end{proof}
Although the notions of $\ind$-Kim-dividing and weak $\ind$-Kim-dividing have different properties (cf., \thref{left-extension-when-kim-fork-eq-kim-div}), one of the goals of this paper is to show that both notions can equivalently be used in developing the structure theory of several neostability-theoretic classes via independence relations. We will make this point explicitly in the next section (\thref{tuwp-is-guwp-with-weak-kim-dividing,tuwp-and-guwp}).

We can also provide an explicit criterion for the equivalence of weak $\ind$-Kim-dividing and $\ind$-Kim-forking at the level of formulas:
\begin{lemma}\thlabel{weak-kim-div-and-kim-fork}
    Let $\ind$ be an independence relation satisfying monotonicity. The following are equivalent:
    \begin{enumerate}[(i)]
        \item If $\phi(x,b)$ $\ind$-Kim-forks over $M$, then it weakly $\ind$-Kim-divides over $M$. 
        \item If $p(x) \in S(Mb)$ and $p(x)$ $\ind$-Kim-forks over $M$, then $p(x)$ $\ind$-Kim-divides over $M$.  
    \end{enumerate}
\end{lemma}
\begin{proof}
    ((i) $\Rightarrow$ (ii)) Suppose $p(x)$ $\ind$-Kim-forks over $M$. By \thref{sfc-for-ind-plus} and definition, there is some $\phi(x,b) \in p(x)$ that $\ind$-Kim-forks over $M$. By (i), $\phi(x,b)$ weakly $\ind$-Kim-divides over $M$. Since $\phi(x,b) \in p(x)$, this means that $p(x)$ $\ind$-Kim-divides over $M$. 

    ((ii) $\Rightarrow$ (i)) Assume $\phi(x,b)$ does not weakly $\ind$-Kim-divide over $M$. Then there is some $a \models \phi(x,b)$ such that $a \ind^+_M b$, i.e., $\tp(a/Mb)$ does not $\ind$-Kim-divide over $M$. By (ii), $\tp(a/Mb)$ does not $\ind$-Kim-fork over $M$. Hence, by \thref{dividing-for-complete-types,commuting-star-and-sfc}, $\phi(x,b)$ does not $\ind$-Kim-fork over $M$. 
\end{proof}
We will discuss examples of this later on. First, we will establish several basic properties of $\ind$-Kim-independence as defined here. 
\begin{lemma} \thlabel{mon-and-strong-fin-char-are-preserved}
    For any independence relation $\ind$, $\ind$-Kim-independence satisfies monotonicity, strong finite character, existence, right extension, and normality. 
\end{lemma}
\begin{proof}
    This follows directly from combining \thref{properties-of-ind-plus,sfc-for-ind-plus,sfc-properties,ind-star-props}. Existence will follow from \thref{find-stronger-than-relative-kim}.
\end{proof}
\begin{lemma} \thlabel{implications-between-relative-kims}
    Let $\ind^1$ and $\ind^2$ be two independence relations and $M \models T$. If $\ind^1 \implies \ind^2$ and $\phi(x,b)$ (weakly) $\ind^1$-Kim-divides over $M$, then it (weakly) $\ind^2$-Kim-divides over $M$. In particular, $\ind^2$-Kim-independence implies $\ind^1$-Kim-independence. 
\end{lemma}
\begin{proof}
    Suppose $\phi(x,b)$ $\ind^1$-Kim-divides over $M$. Let $(b_i)_{i < \omega}$ be an $\ind^1$-Morley sequence over $M$ such that $\{\phi(x,b_i) : i < \omega\}$ is inconsistent. As $\ind^1 \implies \ind^2$, it follows that $(b_i)_{i < \omega}$ is $\ind^2$-Morley over $M$, and thus $\phi(x,b)$ also $\ind^2$-Kim-divides over $M$.

    For the ``in particular'', we can directly use the previous paragraph together with \thref{properties-of-kim-independence}. For a more abstract proof, observe that, since $\ind^1 \implies \ind^2$, we have $(\ind^2)^+ \implies (\ind^1)^+$, and so by \thref{ind-star-props}(ii), $((\ind^2)^+)^* \implies ((\ind^1)^+)^*$, that is, $\ind^2$-Kim-independence implies $\ind^1$-Kim-independence. The implication between the ``weak'' versions of $\ind$-Kim-dividing admits of a similarly abstract proof.
\end{proof}
\begin{example}
    $\find$-Kim-independence $\implies$ Kim-independence $\implies \uind$-Kim-independence. 
\end{example}
The following is an example of a result that seems to require the assumption that ``$\ind$-Kim-forking = $\ind$-Kim-dividing'' and not just the weaker ``$\ind$-Kim-forking = weakly $\ind$-Kim-dividing'':
\begin{lemma} \thlabel{left-extension-when-kim-fork-eq-kim-div}
    Let $\ind$ be an independence relation satisfying full existence and monotonicity. Suppose that, for any formula $\phi(x,b)$ and any $M \models T$, $\phi(x,b)$ $\ind$-Kim-forks over $M$ iff it $\ind$-Kim-divides over $M$. Then $\ind$-Kim-independence satisfies left extension. 
\end{lemma}
\begin{proof}
    We adapt the proof of \cite[Claim 3.20]{chernikov2012forking}. Suppose that $a$ is $\ind$-Kim-independent from $b$ over $M$, and let $c$ be a tuple. By \thref{properties-of-kim-independence}(iii), it suffices to show that there is $c' \equiv_{Ma} c$ such that $\tp(c'a/Mb)$ does not $\ind$-Kim-divide over $M$. Let $p(x,a) := \tp(c/Ma)$. We will show that the set
    \begin{equation*}
        p(x,a) \cup \{\neg \phi(x,a',b') : \phi \in \mathcal{L}(M), \: a'\subseteq a, \: b' \subseteq b,\: \phi(x,y,b') \ind\text{-Kim-divides over } M\}
    \end{equation*}
    is consistent.

    Assume not, for contradiction. By compactness, $p(x,a) \vdash \bigvee_{i < \ell} \phi_i(x,a',b')$ and each $\phi_i(x,y,b')$ $\ind$-Kim-divides over $M$. Let $\psi(x,y,b) := \bigvee_{i < \ell} \phi_i(x,y,b')$. Then $\psi(x,y,b)$ $\ind$-Kim-forks over $M$, so by our assumption it $\ind$-Kim-divides over $M$. 

    Let $I = (b_i)_{i < \omega}$ be an $\ind$-Morley sequence over $M$ such that $\{\psi(x,y,b_i) : i < \omega\}$ is inconsistent. Since $a$ is $\ind$-Kim-independent from $b$ over $M$, in particular we have $a \ind_M^+ b$. So there exists a sequence $I'' = (b''_i)_{i < \omega}$ with $I'' \equiv_{Mb} I$ such that $I''$ is $Ma$-indiscernible. But $p(x,a) \vdash \psi(x,a,b)$. So, by indiscernibility, $p(x,a) \vdash \{\psi(x,a,b''_i) : i < \omega\}$, and thus $p(x,a)$ is inconsistent, a contradiction. 
\end{proof}
\begin{lemma} \thlabel{find-stronger-than-relative-kim}
    For any independence relation $\ind$, we have $\find \implies \ind$-Kim-independence.
\end{lemma}
\begin{proof}
    By \thref{ind-star-props}(v), it suffices to show that $\dind \implies \ind^+$. But this follows from \cite[Lemma 1.4]{shelah1980simple} (the basic characterisation of dividing) and the fact that every $\ind$-Morley sequence over $M$ (if it exists) is $M$-indiscernible by definition.
\end{proof}
Let us note that, despite the close relation between forking and the relative notions of Kim-forking introduced above, in general, we cannot fit forking within this framework. There is, however, a large class of theories in which this is possible:
\begin{lemma} \thlabel{ntp2-forking-algebraic-kim-forking}
    If $T$ is $\NTP_2$, then, for all $M \models T$, $\phi(x,b)$ divides over $M$ iff it $\aind$-Kim-divides over $M$. In particular, $\find$ and $\aind$-Kim-independence coincide over models. 
\end{lemma}
\begin{proof}
    Suppose $\phi(x,b)$ divides over $M \models T$. Since $T$ is $\NTP_2$, $\phi(x,b)$ divides along every strictly $M$-invariant Morley sequence over $M$. Since such sequences exist and are in particular $\aind$-Morley, it follows that $\phi(x,b)$ $\aind$-Kim-divides over $M$. 
\end{proof}
\begin{example}\thlabel{triangle-free-graph}
    Let $T$ be the theory of the generic triangle-free graph. In \cite{conant2017forking}, Conant characterises dividing in this theory. His proof (more specifically, the choice of indiscernible sequence in Construction 4.3) shows that $\phi(x,b)$ divides over $C$ iff it $\aind$-Kim-divides over $C$, for any set of parameters $C$. In particular, the converse to the above lemma does not hold.

    This example also illustrates the distinction between $\ind$-Kim-dividing and weak $\ind$-Kim-dividing in general. By \cite[Theorem 5.3]{conant2017forking}, forking and dividing agree for complete types, and thus, by the previous paragraph and \thref{weak-kim-div-and-kim-fork}, weak $\aind$-Kim-dividing and $\aind$-Kim-forking agree for formulas. \cite[Theorem 6.2]{conant2017forking} shows that $\aind$-Kim-dividing and weak $\aind$-Kim-dividing do not agree.
\end{example}
\begin{example}
    Let $T = T_L^\varnothing$ be the model completion of the empty $L$-theory, where $L = \{f\}$ consists of a binary function symbol. This theory is known to be strictly $\NSOP_1$ (see \cite{kruckman2018generic}). Let $\phi(x, a,b) := f(x,a) = b$, and let $M \models T$ be a model such that $a,b \notin M$. 
    
    Let $(a_ib_i)_{i<\omega}$ be an $M$-indiscernible sequence with $a_0 = a$ and $b_0 = b$ such that $b_i \neq b_j$ for all $i \neq j$ and $a_i = a$ for all $i < \omega$. (Note we can indeed find such a sequence.) In this case, $\{\phi(x,a_i,b_i) : i < \omega\} = \{f(x,a) = b_i : i < \omega\}$ is $2$-inconsistent. Now let $(a_i^*b_i^*)_{i<\omega}$ be an $M$-indiscernible $\aind$-Morley sequence over $M$ with $a^*_0 = a$ and $b_0^* = b$. Since $a \notin M$, we have that $a_i^* \neq a_j^*$ for all $i \neq j$. In particular, $\{\phi(x, a_i^*, b_i^*) : i<\omega\} = \{f(x,a_i^*) = b_i^* : i < \omega\}$ is consistent. This shows that $\phi(x,a,b)$ divides but does not $\aind$-Kim-divide over $M$. 
\end{example}
\subsection{Relative Conant-independence}
We now introduce a new independence relation, which can be seen as an ``existential version'' of $\ind^+$ from the previous subsection:
\begin{definition}
    Let $\ind$ be an independence relation. We define $\ind^\times$ as follows: $A \ind_M^\times B$ iff there is an $\ind$-Morley sequence $I$ over $M$ starting at $B$ which is $MA$-indiscernible.
\end{definition}
Observe that concrete instances of this relation have previously appeared in the literature, e.g., \cite{mutchnik2023nsop3}.
\begin{lemma} \thlabel{properties-of-ind-times}
    Let $\ind$ be an independence relation. Then $\ind^\times$ is an independence relation satisfying left normality. If, in addition, $\ind$ satisfies full existence and monotonicity, then $\ind^\times$ also satisfies existence.
\end{lemma}
\begin{proof}
    We first check that $\ind^\times$ is an independence relation. Suppose that $A \ind^\times_M B$ and $\sigma \in \Aut(\M)$. Let $(B_i)_{i < \omega}$ be an $\ind$-Morley sequence over $M$ with $B_0 = B$ which is $MA$-indiscernible. Then $(\sigma(B_i))_{i < \omega}$ is $\ind$-Morley over $\sigma(M)$, $\sigma(B_0) = \sigma(B)$, and the sequence is $\sigma(M)\sigma(A)$-indiscernible. Therefore, $\sigma(A) \ind_{\sigma(M)}^\times \sigma(B)$.

    (Existence) By full existence and monotonicity, there is an $\ind$-Morley sequence $(M_i)_{i<\omega}$ over $M$ with $M_0 = M$. But then, by indiscernibility, $M_i = M$ for all $i < \omega$. So clearly, for any $A$, we have $A \ind^\times_M M$. 

    (Left normality) This is again immediate.
\end{proof}
\begin{lemma}\thlabel{forall-dividing-for-partial-types}
    If $\pi(x,b)$ is $\ind^\times$-free over $M$, then there is an $\ind$-Morley sequence $(b_i)_{i < \omega}$ over $M$ such that $\bigcup_{i < \omega} \pi(x,b_i)$ is consistent. 
\end{lemma}
\begin{proof}
    Suppose $\pi(x,b)$ is $\ind^\times$-free over $M$. By definition, there is some $c \models \pi(x,b)$ such that $c \ind_M^\times b$. Thus, there exists an $\ind$-Morley sequence $(b_i)_{i < \omega}$ over $M$ with $b_0 = b$ that is $Mc$-indiscernible. In particular, $c \models \bigcup_{i < \omega} \pi(x,b_i)$.
\end{proof}
\begin{lemma} \thlabel{forall-dividing-for-complete-types}
    A type $p(x,b) \in S(Mb)$ is $\ind^\times$-free over $M$ iff there exists an $\ind$-Morley sequence $(b_i)_{i < \omega}$ over $M$ with $b_0 = b$ such that $\bigcup_{i < \omega} p(x,b_i)$ is consistent.
\end{lemma}
\begin{proof}
    ($\Rightarrow$) This is the previous lemma.

    ($\Leftarrow$) Suppose there is an $\ind$-Morley sequence $I = (b_i)_{i < \omega}$ over $M$ with $b_0 = b$ such that $\bigcup_{i < \omega}p(x,b_i)$ is consistent. Let $a' \models \bigcup_{i < \omega} p(x,b_i)$. By Ramsey, compactness, an automorphism and the completeness of $p$, we may assume $I$ is $Ma'$-indiscernible. We can then use automorphisms to find such a sequence for any $a \models p$, which implies that $p$ is $\ind^\times$-free over $M$.
\end{proof}
But note that the proof for strong finite character does not extend to this case. By applying compactness, we obtain, for each $\ind$-Morley sequence, a formula that divides along it. But there is no guarantee here that we can find a single formula that works for all of those sequences.

However, in this case, we can characterise $(\ind^\times)^\sfc$ in a very concrete fashion:
\begin{theorem} \thlabel{sfc-of-forall}
    Suppose $\ind$ satisfies monotonicity. If there is an $\ind$-Morley sequence $(b_i)_{i < \omega}$ over $M$ such that the set $\{\phi(x,b_i) : i < \omega\}$ is consistent, then $\phi(x,b)$ is $\ind^\times$-free over $M$. 
\end{theorem}
\begin{proof}
    Suppose that $\phi(x,b)$ is not $\ind^\times$-free over $M$. This means that, for all $a \models \phi(x,b)$, the type $\tp(a/Mb)$ is not $\ind^\times$-free over $M$. 
    
    Fix an $\ind$-Morley sequence $(b_i)_{i < \omega}$ over $M$ with $b_0 = b$. Let $a \models \phi(x,b)$ and denote $p_a(x,y) := \tp(ab/M)$. Then, by \thref{forall-dividing-for-complete-types}, the set $\bigcup_{i < \omega} p_a(x,b_i)$ is inconsistent. So, by compactness, there is a formula $\psi_a(x,b_a) \in \tp(a/Mb)$ with $b_a \subseteq b$ such that $\{\psi_a(x,b_{a,i}) : i < \omega\}$ is inconsistent, where as usual $b_{a,i}$ denotes the restriction of $b_i$ to the appropriate variables. Now, the set $\{\phi(x,b)\} \cup \{\neg \psi_a(x,b_a) :a 
    \models \phi(x,b)\}$ is inconsistent by our choice of $\psi_a$'s. Thus, by compactness, there exist realisations $a_0, \dots, a_{n-1} \models \phi(x,b)$ such that
    \begin{equation*}
        \phi(x,b) \vdash \bigvee_{j < n} \psi_{a_j}(x,b_{a_j}).
    \end{equation*}
    Assume now, for contradiction, that $\{\phi(x,b_i) : i < \omega\}$ is consistent. By the above and pigeonhole, there is an infinite subset $I \subseteq \omega$ and some $j < n$ such that $\{\psi_{a_j}(x,b_{a_j,i}) : i \in I\}$ is consistent. So, by $M$-indiscernibility, $\{\psi_{a_j}(x,b_{a_j,i}) : i < \omega\}$ is consistent, which contradicts our choice of $\psi_{a_j}$. 

    Since the sequence was arbitrary, we have shown that $\{\phi(x,b_i) : i < \omega\}$ is inconsistent for all $\ind$-Morley sequences $(b_i)_{i < \omega}$ over $M$ with $b_0 = b$.
\end{proof}
\begin{definition}
Let $\ind$ be an independence relation.
\begin{enumerate}[(i)]
    \item We define \textbf{$\ind$-Conant-independence} to be $((\ind^\times)^\sfc)^*$.
    \item We say $\phi(x,b)$ \textbf{$\ind$-Conant-divides over $M$} if it is not $\ind^\times$-free over $M$.
    \item We say $\phi(x,b)$ \textbf{$\ind$-Conant-forks over $M$} if it is not $(\ind^\times)^*$-free over $M$.
\end{enumerate}
\end{definition}
\begin{lemma}\thlabel{conant-ind-defined-right}
    Let $\ind$ be an independence relation satisfying monotonicity.
    \begin{enumerate}[(i)]
        \item A formula $\phi(x,b)$ $\ind$-Conant-divides over $M$ iff, for every $\ind$-Morley sequence $(b_i)_{i < \omega}$ over $M$ with $b_0 = b$, the set $\{\phi(x,b_i) : i < \omega\}$ is inconsistent.
        \item A formula $\phi(x,b)$ $\ind$-Conant-forks over $M$ iff there exist formulae $\psi_i(x,d_i)$ for $i < n$ such that $\phi(x,b) \vdash \bigvee_{i < n} \psi_i(x,d_i)$ and each $\psi_i(x,d_i)$ $\ind$-Conant-divides over $M$. 
        \item We have $A$ is $\ind$-Conant-independent from $B$ over $M$ iff, for some (equiv., any) enumeration $a$ of $A$, $\tp(a/MB)$ does not contain any formulas $\ind$-Conant-forking over $M$. 
    \end{enumerate}
\end{lemma}
\begin{proof}
    (i) This restates \thref{forall-dividing-for-partial-types,sfc-of-forall}.

    (ii) We have that $\phi(x,b)$ $\ind$-Conant-forks over $M$ iff $\phi(x,b)$ is not $((\ind^\times)^\sfc)^*$-free over $M$. By \thref{sfc-star-free-formula}, this holds iff there are formulae $\psi_i(x,d_i)$ for $i<n$ such that $\phi(x,b) \vdash \bigvee_{i < n} \psi_i(x,d_i)$ and no $\psi_i(x,d_i)$ is $\ind^\times$-free over $M$. But by definition $\psi_i(x,d_i)$  is not $\ind^\times$-free over $M$ iff it $\ind$-Conant-divides over $M$. This proves the result.

    (iii) This is \thref{sfc-star-criterion} and (i).
\end{proof}
\begin{example} \thlabel{conant-independence}
    \begin{enumerate}[(i)]
        \item We refer to $\iind$-Conant-independence simply as \textbf{Conant-independence}. As before, we will later see that this notion coincides with that introduced by Mutchnik with the same name in \cite{mutchnik2024conant}.
        \item Following \cite{kim2022some}, we sometimes refer to $\uind$-Conant-independence as \textbf{coheir-Conant-independence}.
    \end{enumerate}
\end{example}
\begin{remark}
    As before, we extend the terminology to complete types as follows: we say $p(x) \in S(Mb)$ \textbf{$\ind$-Conant-divides} (resp., \textbf{$\ind$-Conant-forks}) \textbf{over $M$} if it is not $(\ind^\times)^\sfc$-free (resp., $((\ind^\times)^\sfc)^*$-free) over $M$. It is immediate from the definition that $p(x)$ $\ind$-Conant-divides over $M$ iff there is a formula $\phi(x,b') \in p(x)$ such that $\phi(x,b')$ $\ind$-Conant-divides over $M$. 
\end{remark}
We have observed in the previous subsection that $\ind$-Kim-forking and $\ind$-Kim-dividing might coincide for complete types without doing so for formulas (see \thref{triangle-free-graph}). An analogous argument to \thref{weak-kim-div-and-kim-fork} shows that this cannot happen for $\ind$-Conant-dividing:
\begin{proposition}
    Let $\ind$ be an independence relation satisfying full existence and monotonicity. Then $\ind$-Conant-forking and $\ind$-Conant-dividing agree for complete types iff they agree for formulas. 
\end{proposition}
\begin{proof}
    Left-to-right is clear by definition. For the converse, suppose that $\phi(x,b)$ $\ind$-Conant-forks over $M$. By definition, it is not $(\ind^\times)^*$-free over $M$, so by \thref{commuting-star-and-sfc}, for every $a \models \phi(x,b)$, $\tp(a/Mb)$ $\ind$-Conant-forks over $M$. By assumption, $\tp(a/Mb)$ $\ind$-Conant-divides over $M$. Hence, $a \: (\nind^\times)^\sfc_M \: b$ for every $a \models \phi(x,b)$. Therefore, $\phi(x,b)$ is not $\ind^\times$-free over $M$, i.e., $\phi(x,b)$ $\ind$-Conant-divides over $M$. 
\end{proof}
For the rest of this subsection, we will focus on proving some preliminary properties of $\ind$-Conant-independence. 
\begin{lemma}\thlabel{mon-and-strong-fin-for-universal-kim}
    For any independence relation $\ind$, $\ind$-Conant-independence satisfies monotonicity, strong finite character, right extension, and normality. Moreover, if $\ind$ satisfies full existence and monotonicity, then $\ind$-Conant-independence satisfies existence, and the following holds:
    \begin{enumerate}
        \item[ ] \textbf{Weak left extension}: If $a$ is $\ind$-Conant-independent from $b$ over $M$ and $c$ is a tuple, then there is some $c'a' \equiv_M ca$ such that $\tp(c'a'/Mb)$ does not $\ind$-Conant-divide over $M$.  
    \end{enumerate}
\end{lemma}
\begin{proof}
    The first four properties follow directly from combining \thref{properties-of-ind-times,sfc-properties,ind-star-props}. For the ``moreover'' part, we note that existence will follow from \thref{ind-kim-implies-universal-ind-kim} and prove weak left extension here.

    (Weak left extension) We adapt the proof from \cite[Lemma 2.1(2)]{mutchnik2023properties}. Suppose that $a$ is $\ind$-Conant-independent from $b$ over $M$. Choose $M' \succ M$ to be a sufficiently saturated extension. By right extension, we can find some $a' \equiv_{Mb} a$ such that $a'$ is $\ind$-Conant-independent from $M'$ over $M$. So we may replace $a$ by $a'$ and assume that $b = M'$.

    \begin{claim}
        For any $c$ such that $\tp(c/M')$ does not $\ind$-Conant-divide over $M$, we have that $c$ is $\ind$-Conant-independent from $M'$ over $M$.
    \end{claim}
    \begin{proof}
        Suppose otherwise. Let $\phi(x,d) \in \tp(c/M')$ be a formula which $\ind$-Conant-forks over $M$. Thus, there exist $\psi_i(x,e_i)$ $\ind$-Conant-dividing over $M$ such that $\phi(x,d) \vdash \bigvee_{i<n} \psi_i(x,e_i)$. By saturation of $M'$, we can find some $e'_0, \dots, e'_{n-1}$ in $M'$ such that $e'_0 \dots e'_{n-1} \equiv_{Md} e_0 \dots e_{n-1}$. Hence, $\psi_i(x, e'_i)$ $\ind$-Conant-divides over $M$, and $\phi(x,d) \vdash \bigvee_{i < n} \psi_i(x, e'_i)$ by invariance. But then, as $\tp(c/M')$ is a complete type over $M'$, it follows that there is some $i < n$ with $\psi_i(x, e'_i) \in \tp(c/M')$, contradicting our initial assumption. \hfill \pushQED{$\qed_{\text{Claim}}$}
    \end{proof}
    Take some $c$. We want to show that there is $c'a' \equiv_{M} ca$ with $a' \equiv_{M'} a$ and $\tp(c'a'/M')$ not $\ind$-Conant-dividing over $M$. By compactness, for $\psi(y,x) \in \tp(ca/M)$ and $\phi(x,d) \in \tp(a/M')$ with $d \subseteq M'$, it is enough to find some $c'a'$ such that $\models \psi(c', a') \wedge \phi(a',d)$, so that $\tp(c'a'/Md)$ does not $\ind$-Conant-divide over $M$.

    Now, the formula $\exists y (\psi(y,x) \wedge \phi(x,d)) \in \tp(a/M')$, and so in particular it does not $\ind$-Conant-divide over $M$. Hence, there is some $\ind$-Morley sequence $I = (d_i)_{i < \omega}$ over $M$ starting at $d$ such that $\{\exists y(\psi(y,x) \wedge \phi(x, d_i)) : i < \omega\}$ is consistent. Let $a'$ be a realisation. By Ramsey, compactness, and an automorphism, we may assume $a'$ is such that $I$ is $Ma'$-indiscernible. In particular, $\models \exists y\:\psi(y, a')$, so let $c'$ be a realisation. Another application of Ramsey and compactness allows us to assume $I$ is $Ma'c'$-indiscernible. 

    So it suffices to show that $\tp(c'a'/Md)$ does not $\ind$-Conant-divide over $M$. Let $e \subseteq d$ be finite. For $i < \omega$, we can find $e_i \subseteq d_i$ such that $(e_i)_{i < \omega}$ starts with $e$ and is $Ma'c'$-indiscernible. Now, as $I$ is $\ind$-Morley, we have
    \begin{equation*}
        d_i \ind_M d_0 \dots d_{i-1}.
    \end{equation*}
    By monotonicity of $\ind$, it follows that
    \begin{equation*}
        e_i \ind_M e_0 \dots e_{i-1}.
    \end{equation*}
    Hence, $(e_i)_{i < \omega}$ is $\ind$-Morley. 
    
    Now take $\theta(x,y, e) \in \tp(c'a'/Md)$. By $Ma'c'$-indiscernibility, $\{\theta(x,y,e_i) : i < \omega\}$ is consistent. Therefore, $\theta(x,y,e)$ does not $\ind$-Conant-divide over $M$.
\end{proof}
\begin{lemma}\thlabel{univ-kim-forking-iff-univ-kim-dividing}
    Suppose $\ind$ satisfies full existence, monotonicity, and left and right extension. Then $\phi(x,b)$ $\ind$-Conant-forks over $M$ iff it $\ind$-Conant-divides over $M$.
\end{lemma}
\begin{proof}
    Again, we adapt the standard proof from the literature (cf., \cite[Proposition 5.2]{mutchnik2025nsop}). Let us first begin with a claim:
    \begin{claim}
        If $\models \phi(x, a) \leftrightarrow \psi(x, a, b)$, then $\phi(x,a)$ $\ind$-Conant-divides over $M$ iff so does $\psi(x, a, b)$. 
    \end{claim}
    \begin{proof}[Proof of Claim]
        ($\Rightarrow$) Suppose $\psi(x,a,b)$ does not $\ind$-Conant-divide over $M$; let $(a_ib_i)_{i < \omega}$ be an $\ind$-Morley sequence over $M$ witnessing this. By monotonicity, $(a_i)_{i < \omega}$ is still $\ind$-Morley over $M$. By assumption, $\{\phi(x, a_i) : i < \omega\}$ is consistent. So $\phi(x,a)$ does not $\ind$-Conant-divide over $M$. 

        ($\Leftarrow$) Suppose $(a_i)_{i < \omega}$ is an $\ind$-Morley sequence over $M$ with $a_0 = a$ such that $\{\phi(x, a_i) : i < \omega\}$ is consistent. By right and left extension for $\ind$, there is an $\ind$-Morley sequence $(a_ib_i)_{i < \omega}$ over $M$ such that $a_0b_0 \equiv_M ab$. Therefore, $\{\psi(x, a_i, b_i) : i < \omega\}$ is consistent, as required. \hfill \pushQED{$\qed_{\text{Claim}}$}
    \end{proof}
    Now suppose that $\phi(x,b)$ $\ind$-Conant-forks but does not $\ind$-Conant-divide over $M$. By the claim, this means that $\phi(x,b) \vdash \bigvee_{i<n} \psi_i(x,b)$ and each $\psi_i(x, b)$ $\ind$-Conant-divides over $M$. Let $(b_i)_{i \in \omega}$ be an $\ind$-Morley sequence over $M$ with $b_0 = b$ such that $\{\phi(x, b_i) : i < \omega\}$ is consistent. Then, by the pigeonhole principle, there is an infinite $I \subseteq \omega$ and some $k < n$ such that $\{\psi_k(x, b_i) : i < \omega\}$ is consistent, a contradiction. 
\end{proof}
\begin{lemma} \thlabel{universal-kim-preserves-implications}
    Let $\ind^1$ and $\ind^2$ be two independence relations satisfying full existence and monotonicity, and let $M \models T$. If $\ind^1 \implies \ind^2$ and $\phi(x,b)$ $\ind^2$-Conant-divides over $M$, then it $\ind^1$-Conant-divides over $M$. In particular, $\ind^1$-Conant-independence implies $\ind^2$-Conant-independence. 
\end{lemma}
\begin{proof}
    Simply observe that, since $\ind^1 \implies \ind^2$, we have $(\ind^1)^\times \implies (\ind^2)^\times$, and thus we get $(((\ind^1)^\times)^\sfc)^* \implies (((\ind^2)^\times)^\sfc)^*$ by \thref{sfc-properties}(iii) and \thref{ind-star-props}(ii), i.e., $\ind^1$-Conant-independence implies $\ind^2$-Conant-independence. 
\end{proof}
\begin{example}
    $\uind$-Conant-independence (what is called ``Conant-independence'' in \cite{mutchnik2025nsop}) implies $\iind$-Conant-independence (what is called ``Conant-independence'' in \cite{mutchnik2024conant} and \thref{conant-independence}). In particular, combining \thref{univ-kim-forking-iff-univ-kim-dividing,universal-kim-preserves-implications}, we recover \cite[Fact 6.1]{mutchnik2024conant}.
\end{example}
\begin{lemma}\thlabel{ind-kim-implies-universal-ind-kim}
    Let $\ind$ be an independence relation satisfying full existence and monotonicity. Then $\ind$-Kim-independence implies $\ind$-Conant-independence. 
\end{lemma}
\begin{proof}
    This is clear by the definition of $\ind$-Kim-dividing and \thref{conant-ind-defined-right}.
\end{proof}

\begin{lemma} \thlabel{aind-and-universal-aind-kim}
    In any theory, $\aind = \aind$-Conant-independence.
\end{lemma}
\begin{proof}
    By \cite[Lemma 2.7]{kruckman2018generic}, we know that, in any theory, if $A \nind_C^{\text{a}} B$, then there is a formula $\phi(x,b) \in \tp(a/BC)$ (where $a$ enumerates $A$) that $\aind$-Conant-divides over $C$. This proves the result.
\end{proof}
\begin{corollary}
    For any independence relation $\ind$ satisfying full existence and monotonicity such that $\ind \implies \aind$, we have $\ind$-Conant-independence $\implies \aind$.
\end{corollary}
\begin{proof}
    Combine \thref{universal-kim-preserves-implications,aind-and-universal-aind-kim}.
\end{proof}
\section{Notions of witnessing}\label{sec:witnessing}
The notion of \textit{witnessing} has become central to the development of neostability theory outside of NIP.\footnote{This does not mean that the notion of witnessing has been absent from the development of NIP theories; see \cite{kaplan2014strict} for important results in this regard.} The literature on neostability theory (e.g., \cite{chernikov2012forking,kaplan2020kim,kaplan2021transitivity}) seems to introduce slightly different formalisations of this notion, meant to capture the idea that an instance of (appropriate) dividing is witnessed by a whole class of ``generic'' sequences. In this section, we will introduce three such formalisations that make explicit those notions one can find in the literature. 

The first approach adapts the notion of \textit{relative Kim's lemma} from \cite{mutchnik2024conant}. We first need to generalise the concept of \textit{class} in the definitions of the dividing order in \cite{yaacov2014independence} and the Kim-dividing order in \cite{mutchnik2024conant}:
\begin{definition} \thlabel{class-of-a-type}
    Fix $M \models T$.
    \begin{enumerate}[(i)]
        \item For a global type $p$, we define the \textbf{class of $p$}, denoted $\cl(p)$, to be the class of $\mathcal{L}(M)$-formulas $\phi(x,y)$ such that, for any $M$-indiscernible Morley sequence $(a_i)_{i \in \omega}$ in $p$ over $M$, the set $\{\phi(x, a_i) : i \in \omega\}$ is inconsistent.
        \item Let $\ind^0$ be an independence relation and fix $r \in S(M)$. We denote by $C_0^r$ the class of global $M$-$\ind^0$-free extensions of $r$. Given $M \models T$ and $r \in S(M)$, we say that a global extension $p$ of $r$ is \textbf{$\leq_0$-greatest} if, for all global $q \in C_0^r$, we have $\cl(q) \subseteq \cl(p)$.
    \end{enumerate}
\end{definition}
\begin{remark}
    In general, despite what the name may suggest, $\leq_0$-greatest extensions of types need not be unique, as different global types may share the same class over $M$.
\end{remark}
\begin{example}
    \begin{enumerate}[(i)]
        \item $\find$: $C_{\text{f}}^r$ is the class of global non-forking extensions of $r$. This is closely related to the \textbf{dividing order} from \cite{yaacov2014independence}.
        \item $\iind$: $C_{\text{K}}^r$ is the class of global $M$-invariant extensions of $r$. This is closely related to the \textbf{Kim-dividing order} from \cite{mutchnik2024conant}.
    \end{enumerate}
\end{example}
If we have independence relations $\ind^1$ and $\ind^2$, we use $C_1^r$ and $\leq_1$ (resp., $C_2^r$ and $\leq_2$) to refer to the previous notions as they apply to $\ind^1$ (resp., $\ind^2$). 
\begin{example}
    Observe that, if $M \models T$ and $p$ is $M$-invariant, then for any $\phi(x,y) \in \mathcal{L}(M)$, we have $\phi(x,y) \in \cl(p)$ iff there is some $M$-indiscernible Morley sequence $(a_i)_{i \in \omega}$ in $p$ over $M$ such that $\{\phi(x,a_i) : i \in \omega\}$ is inconsistent.
\end{example}
\begin{definition}
    Let $\ind^1$ and $\ind^2$ be independence relations satisfying full existence and monotonicity. We say that $\ind^1$ satisfies the \textbf{universal witnessing property} (\textbf{UWP}) with respect to $\ind^2$ if $\ind^1$-free types are $\leq_2$-greatest, i.e., for all $M \models T$ and $r \in S(M)$, if $p$ is a global $M$-$\ind^1$-free extension of $r$, then $p$ is $\leq_2$-greatest.
\end{definition}
We can render the above definition more explicitly (without appealing to classes of types) as follows: $\ind^1$ satisfies UWP w.r.t. $\ind^2$ iff, whenever $\phi(x,b)$ $\ind^2$-Kim-divides over $M$, we have that $\{\phi(x,b_i) : i < \omega\}$ is inconsistent for all global $\ind^1$-$M$-free extensions $q$ of $\tp(b/M)$ and every $M$-indiscernible Morley sequence $(b_i)_{i < \omega}$ in $q$ over $M$ with $b_0 = b$.
\begin{terminology}
    Given an independence relation $\ind^1$, sometimes we say that $\leq_1$ is \textbf{trivial} if $\ind^1$ satisfies UWP w.r.t. $\ind^1$. This is equivalent to saying that, for all $r \in S(M)$ and all $p,q \in C_1^r$, $\cl(p) = \cl(q)$. 
\end{terminology}
\begin{example} \thlabel{examples-of-uwp}
    \begin{enumerate}[(i)]
        \item (\cite[Proposition 2.1]{kim1998forking}, \cite[Theorem 2.4]{kim2001simplicity}) $T$ is simple iff $\leq_{\text{f}}$ is trivial. In particular, $\find$ satisfies UWP w.r.t. $\aind$.
        \item (\cite[Lemma 3.14]{chernikov2012forking}) If $T$ is $\NTP_2$, then $\ind^{\textnormal{ist}}$ satisfies UWP w.r.t. $\aind$.
        \item (\cite[Theorem 3.16]{kaplan2020kim}) $T$ is $\NSOP_1$ iff $\iind$ satisfies UWP w.r.t. $\iind$.
        \item (\cite[Theorem 5.2]{kruckman2018generic}) If $T$ is $\NBTP$, then $\kstind$ has UWP w.r.t. $\iind$.
        \item (\cite[Theorem 1.8]{hanson2023bi}) $T$ is $\NCTP$ iff $\ind^{\textnormal{bi}}$ has UWP w.r.t. $\iind$.
        \item (\cite[Theorem 6.2]{mutchnik2025nsop}) If there is an independence relation $\ind$ satisfying full existence such that $\iind$ satisfies UWP w.r.t $\ind$ and $\ind$-Kim-independence is symmetric, then $T$ is $\NSOP_4$.
    \end{enumerate}
\end{example}
Another possibility for formalising the notion of witnessing is to forgo the restriction to global $\ind$-free extensions and talk directly about $\ind$-Morley sequences. This is closer to the original role that Morley sequences played in simple theories, and also aligns better with the phrasing of results in several contexts. For instance:
\begin{itemize}
    \item In simple theories, one can show that, whenever $a \find_M b$, there is a $\find$-Morley sequence $(a_i)_{i \in \omega}$ over $M$ starting at $a$ that is $Mb$-indiscernible, and from this one deduces symmetry for $\find$.
    \item In $\NSOP_1$ theories, one can show that, whenever $a \kind_M b$, there is a tree Morley sequence $(a_i)_{i \in \omega}$ over $M$ starting at $a$ that is $Mb$-indiscernible (see \cite[Lemma 5.12]{kaplan2020kim}), and from this one deduces symmetry. It was later realised (see \cite{kaplan2021transitivity}) that it is enough to find a $\kind$-Morley sequence over $M$ with the same properties. 
\end{itemize}
\textit{A priori}, there is no explicit appeal to global extensions of types in either of these two cases. This motivates the following notion of witnessing:
\begin{definition}\thlabel{guwp-def}
    Let $\ind^1$ and $\ind^2$ be independence relations satisfying full existence and monotonicity. We say that $\ind^1$ satisfies the \textbf{generalised universal witnessing property} (\textbf{GUWP}) with respect to $\ind^2$ if, whenever $\phi(x,b)$ $\ind^2$-Kim-divides over $M$, it also $\ind^1$-Conant-divides over $M$. 
\end{definition}
\begin{remark}
    GUWP $\implies$ UWP.
\end{remark}
Finally, a different notion of witnessing that has appeared in the literature can be found, e.g., in the Kim-Pillay-style theorem for $\NSOP_1$ theories in \cite{kaplan2020kim}. The noteworthy modification is that we begin with a \textit{type} that $\ind$-Kim-divides, instead of a formula. 
\begin{definition}
    Let $\ind^1$ and $\ind^2$ be independence relations satisfying full existence and monotonicity. We say that $\ind^1$ satisfies the \textbf{type universal witnessing property} (\textbf{TUWP}) with respect to $\ind^2$ if, whenever $M \subset B$ and $p(x) \in S(B)$ $\ind^2$-Kim-divides over $M$, there is a formula $\phi(x,b)\in p(x)$ that $\ind^1$-Conant-divides over $M$. 
\end{definition}
TUWP admits two other important interpretations: as the analogue of GUWP for the notion of weak $\ind$-Kim-dividing; and as an implication between independence relations. 
\begin{lemma}\thlabel{tuwp-is-guwp-with-weak-kim-dividing}
    Suppose that $\ind^1$ and $\ind^2$ are independence relations satisfying full existence and monotonicity. The following are equivalent:
    \begin{enumerate}[(i)]
        \item $\ind^1$ satisfies TUWP w.r.t. $\ind^2$.
        \item If $\phi(x,b)$ weakly $\ind^2$-Kim-divides over $M$, then it also $\ind^1$-Conant-divides over $M$. 
        \item $((\ind^1)^\times)^\sfc \implies (\ind^2)^+$. 
    \end{enumerate}
\end{lemma}
\begin{proof}
    ((i) $\Rightarrow$ (iii)) Suppose that $A \: (\nind^2)^+_M \: B$. Let $p(x) := \tp(a/MB)$, where $a$ enumerates $A$. Then $p(x)$ is not $(\ind^2)^+$-free over $M$, and thus, by \thref{forall-dividing-for-complete-types}, $p(x)$ $\ind^2$-Kim-divides over $M$. By TUWP, there is some formula $\phi(x',b) \in p(x)$ with $x' \subseteq x$ finite that $\ind^1$-Conant-divides over $M$. By definition, $\phi(x',b)$ is not $(\ind^1)^\times$-free over $M$. Therefore, $A \: ((\nind^1)^\times)^\sfc_M \: B$. 

    ((iii) $\Rightarrow$ (ii)) Suppose $\phi(x,b)$ weakly $\ind^2$-Kim-divides over $M$. By definition, for each $a \models \phi(x,b)$, we must have $a \: (\nind^2)^+_M \: b$. By our assumption, $a \: ((\nind^1)^\times)^\sfc_M \: b$. Thus, there is some $\psi_a(x_a,b_a) \in \tp(a/Mb)$ with $x_a \subseteq x$, $b_a \subseteq b$ that is not $(\ind^1)^\times$-free over $M$. Hence, $\psi_a(x_a,b_a)$ $\ind^1$-Conant-divides over $M$. 

    Now, the set $\{\phi(x,b)\} \cup \{\neg \psi_a(x_a,b_a) : a \models \phi(x,b)\}$ is inconsistent. Hence, by compactness, there exist realisations $a_0,\dots,a_{n-1} \models \phi(x,b)$ such that 
    \begin{equation*}
        \phi(x,b) \vdash \bigvee_{i < n} \psi_{a_i}(x_{a_i},b_{a_i}).
    \end{equation*}
    We now use \thref{conant-ind-defined-right}(i). Let $(b_j)_{j < \omega}$ be an $\ind^1$-Morley sequence over $M$ with $b_0 = b$. Assume, for contradiction, that $\{\phi(x,b_j) : j < \omega\}$ is consistent. Let $c$ be a realisation. By the above and pigeonhole, there is some $i < n$ an infinite $I \subseteq \omega$ such that $c \models \{\psi_{a_i}(x_{a_i}, b_{a_i,j}) : j \in I\}$, denoting as usual by $b_{a_i,j}$ the restriction of $b_j$ to the appropriate variables. By $M$-indiscernibility and monotonicity of $\ind^1$, this contradicts the fact that $\psi_{a_i}(x_{a_i},b_{a_i})$ $\ind^1$-Conant-divides over $M$. 

    ((ii) $\Rightarrow$ (i)) Suppose that $\tp(a/Mb)$ $\ind^2$-Kim-divides over $M$. By \thref{dividing-for-complete-types}, $\tp(a/Mb)$ is not $(\ind^2)^+$-free over $M$. Hence, by \thref{sfc-for-ind-plus,sfc-criterion}, there is a formula $\phi(x,b) \in \tp(a/Mb)$ which is not $(\ind^2)^+$-free over $M$, i.e., which weakly $\ind^2$-Kim-divides over $M$. Thus, by our assumption, $\phi(x,b)$ $\ind^1$-Conant-divides over $M$, as required.
\end{proof}
For the remainder of this section, we will describe the relationships between these witnessing notions. We start with noting that, for the purposes of witnessing, the distinction between formulas and types is not important:
\begin{lemma}\thlabel{tuwp-and-guwp}
    Suppose that $\ind^1$ and $\ind^2$ are independence relations satisfying full existence and monotonicity. Then $\ind^1$ satisfies GUWP w.r.t. $\ind^2$ iff it satisfies TUWP w.r.t. $\ind^2$. 
\end{lemma}
\begin{proof}
    ($\Rightarrow$) Immediate. 

    ($\Leftarrow$) If $\phi(x,b)$ $\ind^2$-Kim-divides over $M$, then by \thref{properties-of-kim-independence}(i), $\phi(x,b)$ weakly $\ind^2$-Kim-divides over $M$. Hence, by TUWP and \thref{tuwp-is-guwp-with-weak-kim-dividing}, $\phi(x,b)$ $\ind^1$-Conant-divides over $M$, as required. 
    %
\end{proof}
These results make precise the claim made in the previous section that, for the purposes of neostability theory (centred on consequences of Kim's lemma), we can essentially sidestep the distinction between $\ind$-Kim-dividing and weak $\ind$-Kim-dividing. This makes it possible to fully treat the upcoming results of this paper in terms of $\ind^+$-freeness. It also shows, via \thref{tuwp-is-guwp-with-weak-kim-dividing}, that Kim's lemma can be viewed, at its core, as a relationship at the level of abstract independence relations.

In contrast, it is not immediate that the notions of UWP and GUWP need to coincide for arbitrary independence relations. We will now show that, in fact, for most independence relations considered in the literature, they do coincide. A suggestive remark regarding the relationship between GUWP and UWP comes from the proof of \cite[Proposition 5.13]{kaplan2020kim}: if $(a_i)_{i \in \omega}$ is an $\iind$-Morley sequence over $M$, then by ``the compactness of the space of $M$-invariant types,'' we can find a global $M$-invariant extension $q \supset \tp(a_0/M)$ such that $(a_i)_{i \in \omega}$ is Morley in $q$ over $M$. The purpose of this section is to employ the notion of \textit{quasi-strong finite character}, introduced by Mutchnik in \cite{mutchnik2025nsop} and developed further in \cite{kim2022some}, to generalise these results.

We start by adapting some arguments from \cite{kim2022some}. As in that paper, assuming $\ind$ has full existence, monotonicity and quasi-strong finite character, we can define $\Sigma_{a \ind_M b}(x,y)$ to be the $\subseteq$-maximal set of formulas $\Sigma(x,y)$ such that $a',b' \models \Sigma(x,y)$ iff $a \equiv_M a'$, $b \equiv_M b'$, and $a' \ind_M b'$. The following weakening of \cite[Corollary 3.16]{kim2022some}, with a similar proof, does not require $\NATP$ (we will see in \S9 that this applies to the full result too):
\begin{lemma} \thlabel{type-definable-free-extensions}
    Suppose $\ind$ satisfies full existence, monotonicity, and quasi-strong finite character. Let $M \models T$, $a \models p(x) \in S(M)$, and $\Sigma(x) := \{\protect{\psi(x', d') \in \mathcal{L}(M)} : x' \subseteq x$, $z' \subseteq z$ such that $\psi(x', z') \in \Sigma_{a \ind_M d}(x, z)$ and $d' \subseteq d$ is a tuple corresponding to $z'\}$. Then $p(x) \cup \Sigma(x)$ is consistent, and any completion is $\ind$-free over $M$. 
\end{lemma}
\begin{proof}
    Assume, for contradiction, that $p(x) \cup \Sigma(x)$ is not consistent. So 
    \begin{equation*}
        p(x) \vdash \bigvee_{i<n} \psi_i(x'_i, d'_i),
    \end{equation*}
    where $\neg \psi_i(x'_i, d'_i) \in \Sigma(x)$ for all $i < n$. Let $a \models p(x)$. By full existence, there is $a^* = (a_i^*)$ with $a^* \equiv_M a$ such that $a^* \ind_M d'_{<n}$. Since $a^* \models p(x)$, there is some $i$ such that $\models \psi_i(a_i^*, d'_i)$. Thus, $(a_i^*, d'_i) \not\models \Sigma_{a \ind_M d'_i}(x,z_i)$, and so, it follows that $a^* \nind_M d'_i$. This contradicts monotonicity. 

    Now let $q(x)$ be a completion of $p(x) \cup \Sigma(x)$. Let $a \models q|_{Mb}$. Then, in particular, $a \models \Sigma(x)|_{Mb}$, and thus, $a \models \Sigma_{a \ind_M b}(x, b)$. Therefore, $a \ind_M b$, as required. 
\end{proof}
\begin{proposition} \thlabel{collapse-of-sequences-and-types}
    Suppose $\ind$ satisfies existence, monotonicity, right extension, and quasi-strong finite character, and let $M \models T$. If $(a_i)_{i < \omega}$ is an $\ind$-Morley sequence over $M$, then there is a global $M$-$\ind$-free extension $q \supset \tp(a_0/M)$ such that $(a_i)_{i < \omega}$ is Morley in $q$ over $M$. 
\end{proposition}
\begin{proof}
    We adapt the proof from \cite[Proposition 4.4]{dobrowolski2024correction}. By compactness, we find $a_\omega$ such that $(a_i)_{i \leq \omega}$ is $M$-indiscernible. Let $p(x) := \tp(a_\omega/Ma_{<\omega})$ and let $\Sigma(x)$ be as in the above Lemma (over $M$).
    \begin{claim}
        $p(x) \cup \Sigma(x)$ is consistent. 
    \end{claim}
    \begin{proof}[Proof of Claim]
        For any finite $p'(x) \subseteq p(x)$, there is some $i < \omega$ such that $p'(x)$ contains only parameters from $Ma_{<i}$, and so $a_i \models p'$ by $M$-indiscernibility. Since $a_i \ind_M a_{<i}$, by right extension $p'(x)$ extends to a global $M$-$\ind$-free type $q'(x)$, and any realisation of $q'(x)$ will then be a realisation of $p'(x) \cup \Sigma(x)$. Thus, $p'(x) \cup \Sigma(x)$ is finitely consistent. \hfill \pushQED{$\qed_{\text{Claim}}$}
    \end{proof}
    Let $a' \models p(x) \cup \Sigma(x)$, and set $q'(x) := \tp(a'/\M)$. Hence, $q'(x)$ is a completion of $\tp(a_0/M) \cup \Sigma(x)$, and thus, by \thref{type-definable-free-extensions}, it is $\ind$-free over $M$. Now, if we have some $a'' \equiv_{Ma_{<\omega}} a'$, then there is some $\sigma \in \Aut(\M/Ma_{<\omega})$ such that $\sigma(a'') = a_\omega$. Set $q := \sigma(q')$, so that, by invariance for $\ind$, $q$ is also a global $M$-$\ind$-free extension of $p$ and, for any $a \models q$, we have $a \equiv_{Ma_{<\omega}} a_\omega$.

    For any $i < \omega$, we thus have $a_i \equiv_{Ma_{<i}} a_\omega \equiv_{Ma_{<i}} a$. We therefore have $a_i \models q|_{M a_{<i}}$ for all $i < \omega$, and so $(a_i)_{i < \omega}$ is a Morley sequence in $q$ over $M$.
\end{proof}
\begin{corollary}
    Suppose $\ind$ satisfies existence, monotonicity, right extension, and quasi-strong finite character. Then $\phi(x,b)$ $\ind$-Kim-divides over $M$ iff there is some global $M$-$\ind$-free extension $q$ of $\tp(b/M)$ and some $M$-indiscernible Morley sequence $(b_i)_{i < \omega}$ in $q$ over $M$ with $b_0 = b$ such that $\{\phi(x, b_i) : i < \omega\}$ is inconsistent. 
\end{corollary}
\begin{proof}
    ($\Leftarrow$) Clear, and requires no assumptions. 

    ($\Rightarrow$) Suppose $\phi(x,b)$ $\ind$-Kim-divides over $M$. So there is some $\ind$-Morley sequence $(b_i)_{i \in \omega}$ over $M$ such that $\{\phi(x,b_i) : i < \omega\}$ is inconsistent. By \thref{collapse-of-sequences-and-types}, there is a global $M$-$\ind$-free type $q$ extending $\tp(b/M)$ such that $(b_i)_{i \in \omega}$ is a Morley sequence in $q$ over $M$. 
\end{proof}
\begin{example}
    In particular, this shows that our definition of Kim-independence in \thref{kim-independence} does, indeed, coincide with the definition in \cite{kaplan2020kim} (which was implicitly assumed in \thref{ind-relations}).
\end{example}
\begin{corollary}
    Suppose $\ind^1$ satisfies existence, monotonicity, right extension, and quasi-strong finite character, and $\ind^2$ satisfies full existence. Then $\ind^1$ satisfies UWP w.r.t. $\ind^2$ iff it satisfies GUWP w.r.t $\ind^2$. 
\end{corollary}
\begin{proof}
    ($\Leftarrow$) Clear, since every $M$-indiscernible Morley sequence over $M$ in a global $M$-$\ind$-free extension of $\tp(b/M)$ is in particular $\ind$-Morley over $M$.

    ($\Rightarrow$) Suppose $\phi(x,b)$ $\ind^2$-Kim-divides over $M$. Let $(b_i)_{i \in \omega}$ be an $\ind^1$-Morley sequence over $M$ with $b_0 = b$. By \thref{collapse-of-sequences-and-types}, we can find some global $M$-$\ind^1$-free extension $q \supset \tp(b/M)$ such that $(b_i)_{i \in \omega}$ is Morley in $q$ over $M$. By UWP, $\phi(x,y) \in \cl(q)$, which means that $\{\phi(x, b_i) : i \in \omega\}$ is inconsistent, as required. 
\end{proof}
In particular, since strong finite character implies quasi-strong finite character, $\uind$, $\iind$, $\find$, and $\aind$ all have quasi-strong finite character, and by \thref{mon-and-strong-fin-char-are-preserved,mon-and-strong-fin-for-universal-kim}, this also holds for $\ind$-Kim-independence and $\ind$-Conant-independence for any $\ind$ with full existence and monotonicity. This means that all the results on witnessing in the context of simple, $\NTP_2$, $\NSOP_1$, etc. theories from \thref{examples-of-uwp} can be equivalently formulated in terms of GUWP or TUWP. 

Let us finish this section with some basic results about GUWP, which, by the above, can be easily reformulated in terms of UWP and TUWP:
\begin{lemma} \thlabel{relative-dominating-between-relations}
    Let $\ind^1 \implies \ind^2$ and $\ind^3 \implies \ind^4$ be independence relations satisfying full existence and monotonicity. If $\ind^2$ satisfies GUWP w.r.t. $\ind^4$, then $\ind^1$ satisfies GUWP w.r.t. $\ind^3$.
\end{lemma}
\begin{proof}
    Suppose that $\phi(x,b)$ $\ind^3$-Kim-divides over $M$. By \thref{implications-between-relative-kims}, $\phi(x,b)$ also $\ind^4$-Kim-divides over $M$. So, as $\ind^2$ satisfies GUWP w.r.t. $\ind^4$, it follows that $\phi(x,b)$ $\ind^2$-Conant-divides over $M$, and thus, by \thref{universal-kim-preserves-implications}, it also $\ind^1$-Conant-divides over $M$, as required.
\end{proof}
\begin{example}
Using \thref{examples-of-uwp}, the above lemmas allow us to recover several implications from the classification picture abstractly:
    \begin{enumerate}[(i)]
        \item If $T$ is simple, then it is $\NTP_2$. 
        \item If $T$ is simple, then it is $\NSOP_1$.
        \item If $T$ is $\NSOP_1$, then it is $\NBTP$. 
        \item If $T$ is $\NTP_2$, then it is $\NBTP$.
        \item If $T$ is $\NBTP$, then it is $\NCTP$. 
    \end{enumerate}
\end{example}
\begin{lemma} \thlabel{compatible-kim-ind}
    Let $\ind^1$ and $\ind^2$ be compatible independence relations satisfying full existence and monotonicity such that $\ind^1$ satisfies GUWP w.r.t. $\ind^2$, and let $M \models T$. Then $\phi(x,b)$ $\ind^2$-Kim-divides over $M$ iff it $\ind^1$-Conant-divides over $M$. In particular, $\ind^2$-Kim-independence and $\ind^1$-Conant-independence coincide.
\end{lemma}
\begin{proof}
    ($\Rightarrow$) This follows directly from the definition of GUWP.

    ($\Leftarrow$) Assume that $\phi(x,b)$ $\ind^1$-Conant-divides over $M$. By compatibility, there is a global $M$-$\ind^{1 \wedge 2}$-free extension $q \supset \tp(b/M)$. Let $(b_i)_{i \in \omega}$ be an $M$-indiscernible Morley sequence in $q$ over $M$ with $b_0 = b$. Since $(b_i)_{i \in \omega}$ is $\ind^1$-Morley, by assumption $\{\phi(x, b_i) : i \in \omega\}$ is inconsistent. Since $(b_i)_{i \in \omega}$ is also $\ind^2$-Morley, this means that $\phi(x,b)$ $\ind^2$-Kim-divides over $M$. 
\end{proof} 
\begin{example}
    Using again \thref{examples-of-uwp}, the previous lemma recovers the following results from the literature:
    \begin{enumerate}[(i)]
    \item (\cite[Proposition 2.1]{kim1998forking}) If $T$ is simple, then $\phi(x,b)$ divides over $M$ iff it $\find$-Conant-divides over $M$.
    \item (\cite[Theorem 3.16]{kaplan2020kim}) If $T$ is $\NSOP_1$, then $\phi(x,b)$ Kim-divides over $M$ iff it Conant-divides over $M$. 
    \item (\cite[Lemma 3.14]{chernikov2012forking}) If $T$ is $\NTP_2$, then $\phi(x,b)$ divides over $M$ iff it $\istind$-Conant-divides over $M$.
    \item (\cite{kruckman2024new}) If $T$ is $\NBTP$, then $\phi(x,b)$ Kim-divides over $M$ iff it $\kstind$-Conant-divides over $M$. 
\end{enumerate}
\end{example}
\begin{proposition} \thlabel{rel-wit-implies-kim-fork-eq-kim-div}
    Let $\ind^1$ and $\ind^2$ be compatible independence relations satisfying full existence and monotonicity such that $\ind^1$ also satisfies GUWP w.r.t. $\ind^2$, and let $M \models T$. Then $\phi(x,b)$ $\ind^2$-Kim-forks over $M$ iff it $\ind^2$-Kim-divides over $M$.
\end{proposition}
\begin{proof}
    Suppose that $\phi(x,b)$ $\ind^2$-Kim-forks over $M$. By \thref{properties-of-kim-independence}(ii), we have that $\phi(x,b) \vdash \bigvee_{j < n} \psi_j(x, c_j)$ for some $\psi_j(x,c_j)$ for $j <n$ where each $\psi_j(x, c_j)$ $\ind^2$-Kim-divides over $M$. By compatibility, we can find some $\ind^{1 \wedge 2}$-Morley sequence $(b^i, c_0^i, \dots, c_{n-1}^i)_{i \in \omega}$ over $M$ starting at $(b, c_0, \dots, c_{n-1})$. 

    Assume, for contradiction, that $\phi(x,b)$ does not divide along $I_* := (b_i)_{i \in \omega}$. Then there is some $a \models \{\phi(x, b_i) : i \in \omega\}$, and thus, by the pigeonhole principle, we have $a \models \psi_j(x, c_{j,i})$ for some $j<n$ and infinitely many $i \in \omega$. But, by monotonicity and \thref{compatible-kim-ind}, $\{\psi_j(x, c_{j,i}) : i \in I\}$ is inconsistent for any infinite $I \subseteq \omega$. This gives us the required contradiction. Therefore, $\phi(x,b)$ $\ind^2$-Kim-divides over $M$.
\end{proof}
\begin{example}
    This allows us to recover abstractly the following results from the literature:
    \begin{enumerate}[(i)]
        \item (\cite[Corollary 3.22]{chernikov2012forking}) If $T$ is $\NTP_2$, then $\phi(x,b)$ forks over $M$ iff it divides over $M$. 
        \item (\cite[Corollary 4.9.1]{mutchnik2025nsop}) If $T$ is $\NCTP$, then $\phi(x,b)$ Kim-forks over $M$ iff it Kim-divides over $M$. 
    \end{enumerate}
    In particular, by \thref{left-extension-when-kim-fork-eq-kim-div}, $\find$ satisfies left extension over models in all $\NTP_2$ theories (cf. \cite{chernikov2012forking}), and $\kind$ also satisfies left extension in all $\text{NCTP}$ theories.
\end{example}
\begin{lemma} \thlabel{kim-div-and-univ-wit-for-weaker-ind-relations}
    Suppose that $\ind^1$, $\ind^2$ and $\ind^3$ are pairwise compatible independence relations satisfying full existence and monotonicity such that $\ind^1$ satisfies GUWP w.r.t. $\ind^3$ and $\ind^2 \implies \ind^3$, and let $M \models T$. Then $\phi(x,b)$ $\ind^2$-Kim-divides over $M$ iff it $\ind^3$-Kim-divides over $M$.
\end{lemma}
\begin{proof}
    By \thref{relative-dominating-between-relations}, $\ind^1$ also satisfies GUWP w.r.t. $\ind^2$. Hence, by \thref{compatible-kim-ind}, we have that $\phi(x,b)$ $\ind^2$-Kim-divides over $M$ iff it $\ind^1$-Conant-divides over $M$ iff it $\ind^3$-Kim-divides over $M$.
\end{proof}
\begin{example}
    For instance, this lemma recovers the fact that, if $T$ is $\NTP_2$, then $\phi(x,b)$ divides over $M$ iff it Kim-divides over $M$. 
\end{example}
We include the explicit proof of the following result as an illustration of how we can combine all that has been developed so far. The result itself has been previously generalised to NATP theories in \cite[Theorem 3.10]{kim2022some}, but the proof in that case is purely combinatorial. 
\begin{corollary}
    If $T$ is $\NBTP$, then $\phi(x,b)$ Kim-divides over $M$ iff it $\uind$-Kim-divides over $M$. 
\end{corollary}
\begin{proof}
    Since $\kstind$ satisfies GUWP w.r.t. $\iind$ and $\ustind \implies \kstind$, by \thref{relative-dominating-between-relations} it follows that $\ustind$ also satisfies GUWP w.r.t. $\iind$. (Note that $\ustind$ has full existence by \cite[Corollary 3.16]{kim2022some}.) Hence, by \thref{compatible-kim-ind}, if $\phi(x,b)$ Kim-divides over $M$, then it $\ustind$-Kim-divides over $M$, and since $\ustind \implies \uind$, it follows that $\phi(x,b)$ $\uind$-Kim-divides over $M$.
\end{proof}
Finally, let us establish a connection between witnessing and the notion of \textit{weight} (cf. \cite{onshuus2011dp} for a treatment of this concept outside of stability). The following proof is based on \cite[Theorem 4.9]{chernikov2014theories}:
\begin{lemma}\thlabel{guwp-and-weight}
    Let $\ind^1$ and $\ind^2$ be independence relations satisfying full existence and monotonicity. The following are equivalent:
    \begin{enumerate}[(i)]
        \item $\ind^1$ satisfies GUWP w.r.t. $\ind^2$. 
        \item For every $M \models T$, every $\ind^1$-independent sequence $(a_i)_{i < \size{T}^+}$ over $M$, and any finite tuple $b$, there is some $i < \size{T}^+$ such that $\tp(b/Ma_i)$ does not $\ind^2$-Kim-divide over $M$.
    \end{enumerate}
\end{lemma}
\begin{proof}
    ((i) $\Rightarrow$ (ii)) Suppose that there exist $M \models T$, $b$, and $(a_i)_{i < \size{T}^+}$ such that $(a_i)_{i<\size{T}^+}$ is $\ind^1$-independent over $M$ but $\tp(b/Ma_i)$ $\ind^2$-Kim-divides over $M$ for all $i < \size{T}^+$. By pigeonhole, we may assume that the same formula $\phi(x,y)$ witnesses $\ind^2$-Kim-dividing for all $i < \size{T}^+$. By Erd\H{o}s-Rado and compactness, we can find an $Mb$-indiscernible sequence $(a_i')_{i<\omega}$ based on $(a_i)_{i<\size{T}^+}$. In particular, by monotonicity, $(a_i')_{i<\omega}$ is $\ind^1$-Morley over $M$ and $\{\phi(x,a'_i) : i < \omega\}$ is consistent, but $\phi(x,a'_0)$ $\ind^2$-Kim-divides over $M$. So $\ind^1$ does not satisfy GUWP w.r.t. $\ind^2$. 

    ((ii) $\Rightarrow$ (i)) Suppose that $\ind^1$ does not satisfy GUWP w.r.t. $\ind^2$. So there is a formula $\phi(x,a)$ and an $\ind^1$-Morley sequence $(a_i)_{i<\omega}$ over $M$ with $a_0 = a$ such that $\phi(x,a)$ $\ind^2$-Kim-divides over $M$ but $\{\phi(x,a_i) : i < \omega\}$ is consistent. Stretch the sequence $(a_i)_{i<\omega}$ to an $\ind^1$-independent sequence of length $\size{T}^+$. Take $b \models \{\phi(x,a_i) : i < \size{T}^+\}$. Then clearly $\tp(b/Ma_i)$ $\ind^2$-Kim-divides over $M$ for all $i < \size{T}^+$, as it contains $\phi(x,a_i)$ and $a_i \equiv_M a$. In particular, (ii) fails.
\end{proof}
Let us note for the benefit of the reader that an application of the results on witnessing from this section is given in \S\ref{sec:dichotomies} which does not depend on the contents of \S\S\ref{sec:chain-local-character}-\ref{sec:independence-thm}.
\section{Chain local character} \label{sec:chain-local-character}
The goal for the upcoming sections is to use the technology of witnessing developed in the previous section to generalise the proofs of traditional properties of Kim- and Conant-independence in tame contexts, such as symmetry and transitivity, to their relativised versions. As we will see, although the original proofs tend to be done in concrete contexts (such as $\text{NSOP}_1$), for several traditional properties they can be carried out under minimal assumptions on $\ind$ as long as some form of witnessing is in place. 

Before doing so, we need to introduce a slight variation of the local character property from the introduction. We use the notion as in \cite{dobrowolski2022kim}, although it originates in \cite{kaplan2019local}, and we use the name from \cite{delbee2023axiomatic}. Recall that we say $(M_i)_{i<\kappa}$ is a \textbf{continuous chain} of models if $M_i \models T$ for all $i$, $M_i \subseteq M_j$ for all $i \leq j$, and $M_\alpha = \bigcup_{\beta < \alpha} M_\beta$ for every limit $\alpha$.
\begin{enumerate}
    \item[($*$)] \textit{Chain local character}: For all finite $A$, there is a regular cardinal $\kappa$ such that, for all continuous chains of models $(M_i)_{i < \kappa}$ of $T$ such that $\size{M_i} < \kappa$ for all $i < \kappa$, we have $A \ind_{M_i} M$ for some $i < \kappa$, where $M := \bigcup_{i < \kappa} M_i$. 
\end{enumerate}
The following is folklore:
\begin{fact}\thlabel{local-character-over-models}
    Let $\ind$ be an independence relation. If $\ind$ satisfies local character and right base monotonicity, then it also satisfies chain local character. 
\end{fact}
We also have the following result for $\hind$, defined in \thref{hind}:
\begin{fact}[\protect{cf., \cite[Theorem 3.3.7]{delbee2023axiomatic}}]\thlabel{hind-has-local-character}
    $\hind$ satisfies chain local character. 
\end{fact}
More generally, it satisfies local character and right base monotonicity, among other properties, so a quick application of \thref{local-character-over-models} yields this fact.
\begin{lemma}\thlabel{weak-symmetry}
    Let $\ind$ be an independence relation satisfying full existence, monotonicity, and right extension. If $a \ind_M b$, then $\tp(b/Ma)$ does not $\ind$-Conant-divide over $M$. 
\end{lemma}
\begin{proof}
    We adapt the proof from \cite[Proposition 3.22]{kaplan2020kim}. So suppose that $a \ind_M b$. By right extension, we can find an arbitrarily long sequence $(a_i)_{i < \kappa}$ such that $a_i \ind_M ba_{<i}$ for all $i < \kappa$. Applying Erd\H{o}s-Rado, compactness, monotonicity and an automorphism, we obtain an $Mb$-indiscernible $\ind$-Morley sequence $(a'_i)_{i<\omega}$ with $a'_0 = a$. Therefore, $\tp(b/Ma)$ does not $\ind$-Conant-divide over $M$. 
\end{proof}
\begin{remark}
    The proof of \thref{weak-symmetry} shows abstractly that, for any independence relation $\ind$ satisfying full existence and monotonicity, we have $\ind^* \implies (\ind^{\text{opp}})^\times$. In particular, if $\ind$ satisfies right extension and symmetry, we have $\ind \implies \ind^\times$.
\end{remark}
\begin{proposition} \thlabel{local-character-and-coheirs}
    Let $\ind$ be an independence relation satisfying full existence, monotonicity, and right extension. Suppose, in addition, that $\uind \implies \ind$. Then $\ind$-Conant-independence satisfies chain local character.  
\end{proposition}
\begin{proof}
    Let us first note that, if $a \hind_M b$, then $b \uind_M a$ by \thref{hind}, so $b \ind_M a$ by assumption, hence $\tp(a/Mb)$ does not $\ind$-Conant-divide over $M$ by \thref{weak-symmetry}. So $\hind$ implies non-$\ind$-Conant-dividing. Since $\hind$ satisfies right extension, by \thref{ind-star-props}(ii) it follows that $a$ is $\ind$-Conant-independent from $b$ over $M$. 

    Now, let $\kappa$, $(M_i)_{i<\kappa}$, $M$ and $a$ be as in the statement of chain local character. By \thref{hind-has-local-character}, there is $i < \kappa$ such that $a \ind_{M_i} M$. Hence, by the previous paragraph, $a$ is $\ind$-Conant-independent from $M$ over $M_i$. Therefore, $\ind$-Conant-independence satisfies chain local character. 
\end{proof}
It is well-known that we can remove the assumption of $\uind \implies \ind$ by using abstract properties of independence relations. Using \cite[Lemma 9.5]{dobrowolski2022kim}, we obtain:
\begin{proposition}\thlabel{local-character-for-universal-kim-ind}
    Let $\ind$ be an independence relation that satisfies full existence, monotonicity, strong finite character, and right extension. Then $\ind$-Conant-independence satisfies chain local character.
\end{proposition}
\begin{example}
    In any theory, Conant-independence satisfies chain local character. 
\end{example}
\begin{remark}
    In a more abstract fashion, we can use \thref{sfc-properties,sfc-preserves-right-ext} to rephrase the above result as follows: if $\ind$ satisfies existence, monotonicity, and right extension, then $\ind^\sfc$-Conant-independence satisfies chain local character.
\end{remark}
While chain local character, as we have seen above, is a pervasive property among relative Conant-independence, it is less common for relative Kim-independence. A way in which we can obtain it is via a witnessing property:
\begin{lemma}\thlabel{local-character-if-witnessing}
    Let $\ind^1$ and $\ind^2$ be compatible independence relations satisfying full existence and monotonicity such that $\ind^1$ also satisfies strong finite character, right extension, and GUWP w.r.t. $\ind^2$. Then $\ind^2$-Kim-independence satisfies chain local character. 
\end{lemma}
\begin{proof}
    Combine \thref{local-character-for-universal-kim-ind,compatible-kim-ind}.
\end{proof}
\begin{example}
    With this and \thref{examples-of-uwp}, we recover the following results from the literature:
    \begin{enumerate}[(i)]
        \item (Corollary to \cite[Theorem 2.4]{kim2001simplicity}) If $T$ is simple, then $\find$ satisfies chain local character. 
        \item (\cite[Corollary 4.6]{kaplan2019local}) If $T$ is $\NSOP_1$, then Kim-independence satisfies chain local character.
    \end{enumerate}
\end{example}
\begin{remark}
    Note that the above lemma does not extend to $\find$ and $\kind$ in $\NTP_2$ and $\NBTP$ theories, respectively, because outside of simple and $\NSOP_1$ theories, $(\find)^{\text{opp}}$ and $(\kind)^{\text{opp}}$ do not satisfy strong finite character.
\end{remark}

A different way to obtain chain local character for relative Kim-independence which avoids witnessing is via symmetry. 
\begin{lemma}\thlabel{symmetry-implies-local-character}
    Let $\ind$ be an independence relation satisfying full existence and monotonicity. If $\ind$-Kim-independence is symmetric, then it satisfies chain local character. 
\end{lemma}
\begin{proof}
    Suppose $\ind$-Kim-independence is symmetric. By \thref{mon-and-strong-fin-char-are-preserved} and full existence and monotonicity for $\ind$, $\ind$-Kim-independence also satisfies strong finite character, monotonicity, and existence. Therefore, by \cite[Lemma 9.5]{dobrowolski2022kim}, $\uind \implies \ind$-Kim-independence, and thus, by symmetry, $\hind \implies \ind$-Kim-independence. Hence, as before, $\ind$-Kim-independence satisfies chain local character. 
\end{proof}
\section{Symmetry}\label{sec:symmetry}
In this section, we will make explicit the formal connection between witnessing and symmetry that appears in contexts such as simplicity and $\text{NSOP}_1$. This result forms the basis for more complicated properties, such as relativisations of the Independence Theorem.
\begin{theorem}\thlabel{symmetry-characterisation}
    Let $\ind$ be an independence relation satisfying full existence and monotonicity. Suppose that $\ind$-Kim-independence is compatible with $\ind$ and satisfies left transitivity. The following are equivalent:
    \begin{enumerate}[(i)]
        \item $\ind$-Kim-independence satisfies GUWP w.r.t. $\ind$.
        \item $\ind$-Kim-independence is symmetric. 
    \end{enumerate}
\end{theorem}
\begin{proof}
   ((i) $\Rightarrow$ (ii)) Suppose that $a$ is $\ind$-Kim-independent from $b$ over $M$. By \thref{weak-symmetry}, $\tp(b/Ma)$ does not $\ind^{\text{K}^{\ind}}$-Conant-divide over $M$. Thus, by GUWP, $\tp(b/Ma)$ does not $\ind$-Kim-divide over $M$, and hence, by compatibility and  \thref{rel-wit-implies-kim-fork-eq-kim-div}, $b$ is $\ind$-Kim-independent from $a$ over $M$.

   ((ii) $\Rightarrow$ (i)) We adapt the argument from \cite{kaplan2021transitivity}. Recall (from \cite{dobrowolski2022kim} that we say a continuous chain of models $(M_i)_{i<\kappa}$ is an \textbf{$\ind$-independence chain} for $(a_i)_{i<\kappa}$ over $M$ if $M \subseteq M_0$, $a_{<i} \subset M_i$, and $a_i \ind_M M_i$ for all $i < \kappa$. Let $\phi(x,b)$ be a formula that $\ind$-Kim-divides over $M$. Let $(b_i)_{i<\omega}$ be an $\ind^{\text{K}^{\ind}}$-Morley sequence over $M$ with $b_0 = b$. Using right extension, we may extend the sequence to length $\kappa$ for a sufficiently large $\kappa$.

    Assume, for contradiction, that $\{\phi(x,b_i) : i < \kappa\}$ is consistent. Let $a$ be a realisation. Since it is symmetric, by \thref{symmetry-implies-local-character} $\ind$-Kim-independence satisfies chain local character; we may assume $\kappa$ is the cardinal coming out of this property for the finite tuple $a$. Using existence and right extension, we can find an $\ind^{\text{K}^{\ind}}$-independence chain of models $(M_i)_{i<\kappa}$ for $(b_i)_{i<\kappa}$ over $M$ such that $\size{M_i} < \kappa$. By chain local character, if we let $M_{\kappa} := \bigcup_{i < \kappa} M_i$, then $a$ is $\ind$-Kim-independent from $M_{\kappa}$ over $M_i$ for some $i < \kappa$. So, by monotonicity, $a$ is $\ind$-Kim-independent from $b_i$ over $M_i$. But, by choice of the continuous chain of models, $b_i$ is $\ind$-Kim-independent from $M_i$ over $M$, so by symmetry and transitivity it follows that $M_ia$ is $\ind$-Kim-independent from $b_i$ over $M$. This contradicts the fact that $\phi(x,b_i) \in \tp(a/Mb_i)$ and it $\ind$-Kim-forks over $M$ by assumption. 

    Hence, $\{\phi(x,b_i) : i < \kappa\}$ is inconsistent, and thus, by indiscernibility, it is $k$-inconsistent for some $k < \omega$. Therefore, $\{\phi(x,b_i) : i < \omega\}$ is inconsistent, which means that $\phi(x,b)$ $\ind^{\text{K}^{\ind}}$-Conant-divides over $M$.
\end{proof}
\begin{remark}
    The condition of compatibility between $\ind$ and $\ind$-Kim-independence is a technical assumption that is easy to prove in most concrete cases. For example, one can easily see from \thref{find-stronger-than-relative-kim} that it suffices to have compatibility between $\ind$ and $\find$. In particular, $\uind$, $\iind$, $\find$ and $\aind$ all satisfy this condition.
\end{remark}
\begin{example}\thlabel{find-symmetry-simple}
    By Kim's Lemma for simple theories \cite[Proposition 2.1]{kim1998forking}, if $T$ is simple, then $\find$ satisfies GUWP w.r.t. $\aind$ and $\find = \aind$-Kim-independence. Since $\find$ always satisfies left transitivity (\thref{ind-relations}), the above result recovers symmetry of $\find$ \cite[Theorem 2.5]{kim1998forking}.
\end{example}
One immediate consequence of the above result is a new characterisation of simplicity within $\NTP_2$ theories. This result is the analogue of \cite[Proposition 8.7]{kaplan2020kim}, which similarly characterises simplicity within $\NSOP_1$ theories.
\begin{corollary}
    An $\NTP_2$ theory is simple iff $\istind$-Kim-independence satisfies GUWP w.r.t. $\istind$.
\end{corollary}
\begin{proof}
    By \thref{examples-of-uwp,find-stronger-than-relative-kim,ind-kim-implies-universal-ind-kim}, if $T$ is $\NTP_2$, then $\find = \istind$-Kim-independence and so it satisfies left transitivity. Thus, if $\istind$-Kim-independence satisfies GUWP w.r.t. $\istind$, then by \thref{symmetry-characterisation} $\find$ is symmetric, and thus $T$ is simple. 
    
    Conversely, if $T$ is simple, then $\find = \istind$-Kim-independence is symmetric, and therefore by \thref{symmetry-characterisation} $\istind$-Kim-independence satisfies GUWP w.r.t. $\istind$.
\end{proof}
Another application of our result provides an abstract, purely semantic proof of one of the implications of \cite[Theorem 3.12]{mutchnik2024conant}, without the need to use the technical tree machinery of the original proof. Let us recall from \cite[Definition 3.3]{mutchnik2024conant} that $\ind$ satisfies the \textbf{generalised freedom axiom} if, whenever $M \prec N \models T$ and there is an $N$-indiscernible $\ind$-Morley sequence over $M$ starting at $a$, then every $\ind$-Morley over $N$ starting at $a$ is also $\ind$-Morley over $M$.
\begin{lemma} \thlabel{stat-and-ind-kim-forking}
    Suppose $\ind$ satisfies full existence, monotonicity, and stationarity over models. Then $\ind$-Kim-forking and $\ind$-Kim-dividing coincide.
\end{lemma}
\begin{proof}
    This is \cite[Proposition 3.10]{mutchnik2024conant}. Within our framework, this result can be obtained from two facts: (i) given our assumptions, $\ind$ also satisfies left and right extension, which means that (ii) by \thref{univ-kim-forking-iff-univ-kim-dividing}, $\ind$-Conant-forking and $\ind$-Conant-dividing coincide. Since stationarity implies that $\ind$-Kim-independence and $\ind$-Conant-independence agree, the result follows.
\end{proof}
\begin{lemma} \thlabel{left-transitivity-using-generalised-freedom}
    Suppose $\ind$ satisfies full existence, monotonicity, stationarity over models, and the generalised freedom axiom. Then $\ind$-Kim-independence satisfies left transitivity. 
\end{lemma}
\begin{proof}
    Suppose that $M \prec N$ are models of $T$, $N$ is $\ind$-Kim-independent from $a$ over $M$, and $b$ is Kim-independent from $a$ over $N$. In particular, by full existence and monotonicity for $\ind$, there exists an $N$-indiscernible $\ind$-Morley sequence over $M$ starting at $a$. Similarly, there exists an $Nb$-indiscernible $\ind$-Morley sequence $I$ over $N$ starting at $a$. In particular, by the generalised freedom axiom, $I$ is also $\ind$-Morley over $M$, and since it is $Nb$-indiscernible, it is also $Mb$-indiscernible. Since, by stationarity, $\tp(a/M)$ has a unique global $M$-$\ind$-free extension $q$ and $I$ is Morley in $q$ over $M$ by \thref{collapse-of-sequences-and-types}, it follows by definition and \thref{stat-and-ind-kim-forking} that $b$ is $\ind$-Kim-independent from $a$ over $M$, as claimed.
\end{proof}
\begin{corollary}
    Let $\ind$ be an independence relation satisfying full existence, monotonicity, stationarity over models, and the generalised freedom axiom. If $\ind$-Kim-independence is symmetric, then $\iind$ satisfies GUWP w.r.t. $\ind$. 
\end{corollary}
\begin{proof}
    By \thref{left-transitivity-using-generalised-freedom}, $\ind$-Kim-independence satisfies left transitivity. So, by \thref{symmetry-characterisation} (observe that the compatibility condition holds in this case), if $\ind$-Kim-independence is symmetric, then $\ind$-Kim-independence satisfies GUWP w.r.t. $\ind$. As $\iind \implies \ind$-Kim-independence, it follows from \thref{relative-dominating-between-relations} that $\iind$ satisfies GUWP w.r.t. $\ind$. 
\end{proof}
In this way, we recover half of \cite[Theorem 3.12]{mutchnik2024conant}. The other direction requires further tools; \thref{symmetry-characterisation} can only prove the weakening that involves $\ind$-Kim-independence instead of $\iind$. 

Let us finish this section by discussing further the relationship between Adler's notion of an \textit{independence relation} and the notion of witnessing that we have developed here. Let us start with the following definition (we adapt the terminology from \cite{adler2008introduction,casanovas2011simple,delbee2023axiomatic}):
\begin{definition}
    We say $\ind$ is an \textbf{Adler preindependence relation} if it satisfies monotonicity, right base monotonicity, left transitivity, left normality, and strong finite character.

    We say $\ind$ is an \textbf{Adler independence relation} if it satisfies monotonicity, right base monotonicity, left transitivity, left normality, finite character, local character, and right extension.
\end{definition}
A well-known theorem of Adler says the following:
\begin{fact}[\protect{\cite[Theorem 2.5]{adler2009geometric}}] \thlabel{adler}
    Every Adler independence relation is symmetric. 
\end{fact}
The proof of this result sheds some light on the relationship of this notion with relative Kim- and Conant-independence:
\begin{lemma} \thlabel{air-and-existence-of-morley-seqs}
    Suppose $\ind$ is an Adler independence relation. Then $\ind = \ind^\times$, i.e., $a \ind_M b$ iff there exists an $Ma$-indiscernible $\ind$-Morley sequence over $M$ that starts at $b$. 
\end{lemma}
\begin{proof}
    ($\Rightarrow$) Suppose $a \ind_M b$. By \thref{adler}, $\ind$ satisfies left extension, so we can build inductively a sufficiently long sequence $(b_i)_{i < \kappa}$ with $b_0 = b$ such that $b_{<i} a \ind_M b_i$ for all $i < \kappa$. By Erd\H{o}s-Rado, compactness, monotonicity and an automorphism, there is an $\ind^{\text{opp}}$-Morley sequence $(b'_i)_{i < \omega}$ locally based on $(b_i)_{i<\kappa}$ and starting at $b$ which is $Ma$-indiscernible. By symmetry, $(b'_i)_{i < \omega}$ is the sequence we wanted.

    ($\Leftarrow$) Suppose that there exists an $Ma$-indiscernible $\ind$-Morley sequence $(b_i)_{i < \omega}$ with $b_0 = b$. Then it follows directly from \cite[Proposition 2.4]{adler2009geometric} and symmetry that $a \ind_M b$. 
\end{proof}
\begin{example}
    The above lemma strengthens \cite[Theorem 4.1.7]{onshuus2006properties} for thorn-forking in the context of rosy theories.
\end{example}
\begin{corollary}
    Suppose that $\ind$ is an Adler independence relation satisfying strong finite character. Then $\ind = \ind$-Conant-independence over models. 
\end{corollary}
\begin{proof}
    By \thref{air-and-existence-of-morley-seqs,commuting-star-and-sfc}, we have $\ind = \ind^* = (\ind^*)^\sfc = (\ind^\sfc)^* = ((\ind^\times)^\sfc)^* = \ind$-Conant-independence.
\end{proof}
It remains an open question (see \cite[Question 1.1]{adler2009thorn}) whether every Adler independence relation satisfies strong finite character. The above result shows that this condition is equivalent to the property called ``witnessing'' in \cite{kaplan2020kim}.
\begin{theorem} \thlabel{air-and-guwp}
    Let $\ind$ be an Adler independence relation. The following are equivalent:
    \begin{enumerate}[(i)]
        \item $\ind = \ind$-Kim-independence over models. 
        \item $\ind$ satisfies GUWP w.r.t. $\ind$.
    \end{enumerate}
\end{theorem}
\begin{proof}
    ((i) $\Rightarrow$ (ii)) Assume $\ind = \ind$-Kim-independence. Then $\ind$-Kim-independence satisfies left transitivity and symmetry by \thref{adler}. Therefore, by \thref{symmetry-characterisation}, $\ind$-Kim-independence satisfies GUWP w.r.t. $\ind$. As $\ind = \ind$-Kim-independence, this means that $\ind$ satisfies GUWP w.r.t. $\ind$. 

    ((ii) $\Rightarrow$ (i)) If $\ind$ satisfies GUWP w.r.t. $\ind$, then by \thref{compatible-kim-ind} $\ind$-Kim-independence and $\ind$-Conant-independence coincide, and by \thref{rel-wit-implies-kim-fork-eq-kim-div} $\ind$-Kim-dividing and $\ind$-Kim-forking also agree. Thus, by \thref{air-and-existence-of-morley-seqs}, $\ind = \ind$-Kim-independence.
\end{proof}
We thus recover the implication that, if $\dind$ is an Adler independence relation, then $T$ is simple: if $\dind$ is an Adler independence relation, then $\dind = (\dind)^* = \find$ and so by \thref{air-and-existence-of-morley-seqs} $\find$-Kim-independence implies $\find$. Hence, by \thref{find-stronger-than-relative-kim}, $\find = \find$-Kim-independence. Therefore, by \thref{air-and-guwp}, $\find$ satisfies GUWP w.r.t. $\find$, i.e., Kim's Lemma for simple theories holds. So $T$ is simple.
\section{Independence Theorem}\label{sec:independence-thm}
In this section, we want to use the notion of witnessing developed previously to prove a generalised version of the independence theorem. This result has been key to the study of simple and $\NSOP_1$ theories. Our version of the Independence Theorem for relative Kim-independence will require us to briefly explore a property which appears already in \cite{adler2009geometric} and also, in passing, in \cite[\S9.2]{kaplan2020kim}, where it is not given an explicit name but only described as ``a weak form of transitivity.'' We will name it accordingly: 
\begin{enumerate}
    \item[($\star$)] \textit{Weak transitivity\index{transitivity!weak}}: For all $M \models T$ and $A,B,C \subset \M$, if $A \ind_M BC$ and $C \ind_M B$, then $AC \ind_M B$. 
\end{enumerate}
This is a relatively common property of traditional independence relations. 
\begin{remark}
    Let us note that, by \cite[Proposition 3.1.3]{delbee2023axiomatic}, if $\ind$ is an Adler preindependence relation, then $\ind^*$ is also an Adler preindependence relation which, in addition, satisfies right normality.
\end{remark}
\begin{lemma}
    Adler preindependence relations satisfying right normality also satisfy weak transitivity.
\end{lemma}
\begin{proof}
    This is \cite[Remark 2.1]{adler2009geometric}.
\end{proof}
\begin{remark}
    \begin{enumerate}[(i)]
        \item The above does not require strong finite character. 
        \item The use of right normality can be avoided if we redefine weak transitivity by: $A \ind_M BCM$ and $C \ind_M BM$ implies $AC \ind_M B$. However, the definition given above, as well as being the original formulation in \cite{kaplan2020kim}, is also more naturally related to transitivity in the usual sense.
    \end{enumerate}
\end{remark}
\begin{example}
    $\uind$, $\iind$, $\find$, and $\dind$ are all Adler preindependence relations satisfying right normality (see \cite{casanovas2011simple}), so by the above, they all satisfy weak transitivity.
\end{example}
The importance of weak transitivity lies in the fact that it allows us to find new Morley sequences out of old ones:
\begin{lemma}\thlabel{weak-transitivity-and-morley-sequences}
    Let $\ind$ be an independence relation satisfying weak transitivity and $n < \omega$. If $(b_i)_{i<\omega}$ is $\ind$-independent over $M$, then so is $(b_{ni}, b_{ni+1}, \dots, b_{(i+1)n-1})_{i<\omega}$. 
\end{lemma}
\begin{proof}
    Fix $n < \omega$. Let $\bar{b}_i := (b_{in}, b_{in+1}, \dots, b_{(i+1)n-1})$. We prove by induction on $i$ that $\bar{b}_i \ind_M \bar{b}_{<i}$. 

    First, note that, by $\ind$-independence, $b_n \ind_M b_0\dots b_{n-1}$ and $b_{n+1} \ind_M b_0 \dots b_{n}$, so by weak transitivity $b_nb_{n+1} \ind_M b_0 \dots b_{n-1}$. Inductively, for any $1 \leq k \leq n-1$, we have $b_n b_{n+1} \dots b_{n+k} \ind_M b_0 \dots b_{n-1}$, and again by $\ind$-independence, we also have $b_{n+k+1} \ind_M b_0 \dots b_{n+k}$. So, by weak transitivity again, $b_n \dots b_{n+k+1} \ind_M b_0 \dots b_{n-1}$. This proves that $b_n \dots b_{2n-1} \ind_M b_0 \dots b_{n-1}$. 

    Suppose now that we have shown $\bar{b}_i \ind_M \bar{b}_{<i}$ for some $i < \omega$. Then it is easy to adapt the argument from the previous paragraph to show that $\bar{b}_{i+1} \ind_M \bar{b}_{\leq i}$. This concludes the induction.
\end{proof}
Employing this notion, we can turn to the main topic of this section, which is a generalisation of the independence theorem over models in terms of a given independence relation $\ind'$:
\begin{enumerate}
    \item[($\dagger$)] \textit{$\ind'$-independence theorem over models}: If $M \models T$, $a_1 \equiv_M a_2$, $a_1 \ind_M b_1$, $a_2 \ind_M b_2$, and $b_1 \ind'_M b_2$, then there is some $a_*$ such that $a_* \equiv_{Mb_1} a_1$, $\protect{a_* \equiv_{Mb_2} a_2}$, and $a_* \ind_M b_1b_2$. 
\end{enumerate}
This corresponds precisely to what is called ``$\ind^0$-amalgamation over models'' in \cite{delbee2023axiomatic}. It is also clear that we recover \thref{properties-of-independence}(xvii) when we set $\ind'= \ind$. Before relating this notion to witnessing, let us make a few preliminary observations on the interaction between $\ind^+$ and the $\ind$-independence theorem over models. 
\begin{lemma}\thlabel{independence-theorem-and-ind-plus}
    Suppose $\ind^1$ and $\ind^2$ satisfy full existence and monotonicity. If $\ind^1$ satisfies the $\ind^2$-independence theorem, then $\ind^1 \implies (\ind^2)^+$.
\end{lemma}
\begin{proof}
    Suppose that $a \ind^1_M b$. By full existence and monotonicity, we can find an $\ind^2$-Morley sequence $(b_i)_{i < \omega}$ over $M$ with $b_0 = b$. Let $p(x,y) : =\tp(ab/M)$. By assumption, there is some $\sigma \in \Aut(\M/M)$ with $\sigma(b) = b_1$. Let $a' := \sigma(a)$. Then $a' \equiv_M a$, and as $\ind^1$ is an independence relation, $a' \ind_M^1 b_1$. Since $b_1 \ind_M^2 b_0$, it follows from the $\ind^2$-independence theorem that there is some $a''$ such that $a'' \equiv_{Mb_0} a$, $a'' \equiv_{Mb_1} a'$, and $a'' \ind^1_M b_0b_1$. By induction, we can find some $a^*$ such that $a^*b_i \equiv_M ab$ for all $i < \omega$ and $a^* \ind^1_M b_{<\omega}$. In particular, $a^* \models \bigcup_{i < \omega} p(x,b_i)$. Since $(b_i)_{i < \omega}$ was arbitrary, it follows from \thref{dividing-for-complete-types} that $p(x,b)$ is $(\ind^2)^+$-free over $M$, and thus $a \: (\ind^2)^+_M \: b$. 
\end{proof}
\begin{remark}
    \thref{triangle-free-graph} shows that, in general, $\ind^+$ does not satisfy the $\ind$-independence theorem over models. Thus, \thref{independence-theorem-and-ind-plus} cannot be directly used for a characterisation of $\ind^+$ similar to those of $\ind^*$ and $\ind^\sfc$.
\end{remark}
\begin{corollary}
    Suppose $\ind^1$ and $\ind^2$ satisfy full existence and monotonicity. If $\ind^1$-Kim-independence satisfies the $\ind^2$-independence theorem over models, then $\ind^1$-Kim-independence implies $\ind^2$-Kim-independence.
\end{corollary}
\begin{proof}
    If $\ind^1$-Kim-independence satisfies the $\ind^2$-independence theorem, then by \thref{independence-theorem-and-ind-plus} $\ind^1$-Kim-independence $\implies (\ind^2)^+$. Since $\ind^1$-Kim-independence satisfies right extension, $\ind^1$-Kim-independence $\implies ((\ind^2)^+)^*$. But the latter is $\ind^2$-Kim-independence by definition.
\end{proof}
We can now adapt almost word-for-word the proofs of the following important facts to be found in, e.g., \cite{kaplan2020kim}. The main inspiration for this comes from \cite{mutchnik2025nsop}, which constitutes the first generalisation of the independence theorem to a more general version of witnessing outside $\NSOP_1$ theories. 
\begin{lemma}[Chain condition] \thlabel{chain-condition}
    Let $\ind$ be an independence relation satisfying full existence, monotonicity, and weak transitivity. If $\tp(a/Mb)$ does not $\ind$-Kim-divide over $M$ and $I$ is an $\ind$-Morley sequence over $M$ starting at $b$, there is $a'\equiv_{Mb} a$ such that $\tp(a'/MI)$ does not $\ind$-Conant-divide over $M$ and $I$ is $Ma'$-indiscernible. 
\end{lemma}
\begin{proof}
    Note that, by definition, there is $a'\equiv_{Mb} a$ such that $I$ is $Ma'$-indiscernible. By finite character, it suffices to show that $\tp(a'/Mb_{<n})$ does not $\ind$-Conant-divide over $M$ for all $n < \omega$. So fix some $n$. By \thref{weak-transitivity-and-morley-sequences}, $(b_{kn}, b_{kn+1}, \dots, b_{(k+1)n-1})_{k < \omega}$ is $\ind$-Morley over $M$ and clearly $Ma'$-indiscernible. This implies (by definition and \thref{sfc-properties}(ii)) that $\tp(a'/Mb_0\dots b_{n-1})$ does not $\ind$-Conant-divide over $M$. This completes the induction.
\end{proof}
\begin{theorem} \thlabel{wit}
    Let $\ind^1 \implies \ind^2$ be independence relations satisfying existence, monotonicity, and right extension, such that $\ind^1$ satisfies left extension and $\ind^2$ satisfies weak transitivity. Suppose that $\ind^2 \implies \ind^2$-Kim-independence and the latter satisfies symmetry and GUWP w.r.t. $\ind^2$. Then $\ind^2$-Kim-independence satisfies the $\ind^1$-independence theorem over models.
\end{theorem}
\begin{proof}
    We adapt the argument from \cite[Proposition 5.8]{mutchnik2025nsop}, which is itself based on the original version of the weak independence theorem in \cite[Proposition 6.1]{kaplan2020kim}. So let $M\models T$ and suppose that:
    \begin{enumerate}[(i)]
        \item $a_1 \equiv_M a_2$.
        \item $a_1$ is $\ind^2$-Kim-independent from $b_1$ over $M$.
        \item $a_2$ is $\ind^2$-Kim-independent from $b_2$ over $M$. 
        \item $b_1 \ind^1_M b_2$.
    \end{enumerate}
    Let us first note that, by \thref{universal-kim-preserves-implications,compatible-kim-ind} (where we use GUWP to apply the last lemma) together with $\ind^2 \implies \ind^2$-Kim-independence, a formula $\ind^2$-Kim-divides over $M$ iff it $\ind^2$-Conant-divides over $M$. We can also use \thref{rel-wit-implies-kim-fork-eq-kim-div} to show that $\ind^2$-Kim-forking and $\ind^2$-Kim-dividing coincide for formulas.
    \begin{claim}
        There is some $b'_1$ such that $a_1b_1 \equiv_M a_2b_1'$ and $a_2$ is $\ind^2$-Kim-independent from $b_1'b_2$ over $M$.
    \end{claim}
    \begin{proof}[Proof of Claim]
        Since $\ind^2$-Kim-independence is symmetric, it suffices to find $b'_1$ such that $a_1b_1 \equiv_M a_2b'_1$ and $b'_1b_2$ is $\ind^2$-Kim-independent from $a_2$ over $M$. Letting $p(x,a_1) := \tp(b_1/Ma_1)$, by (ii) and symmetry $b_1$ is $\ind$-Kim-independent from $a_1$ over $M$, and so, by (i) and \thref{properties-of-kim-independence}(iii), $p(x,a_2)$ does not contain any formulas $\ind$-Kim-forking over $M$. Hence, it suffices to show the consistency of
        \begin{equation*}
            p(x,a_2) \cup \{\neg \phi(x,b_2^*,a_2^*) : b_2^* \subseteq b_2,\: a_2^* \subseteq a_2, \: \phi(x,y,a_2^*) \: {\ind}^2\text{-Kim-forks over } M\}.
        \end{equation*}
        So assume, for contradiction, that the above set is inconsistent. By compactness, we have that $p(x,a_2) \vdash \phi(x,b_2^*,a_2^*)$ for some finite $b_2^* \subseteq b_2$, $a_2^* \subseteq a_2$ and a formula $\phi(x,y,a_2^*)$ that $\ind^2$-Kim-forks over $M$, and hence $\ind^2$-Kim-divides over $M$. 

        By (iii) and symmetry, $b_2$ is $\ind^2$-Kim-independent from $a_2$ over $M$. So, by definition, full existence and monotonicity for $\ind^2$, there is an $\ind^2$-Morley sequence $(a_{2,i})_{i\in\omega}$ over $M$ starting at $a_2$ which is $Mb_2$-indiscernible. Hence, 
        \begin{equation*}
            \bigcup_{i<\omega} p(x,a_{2,i}) \vdash \{\phi(x,b_2^*,a_{2,i}^*) : i < \omega\}.
        \end{equation*}
        (As usual, $a_{2,i}^*$ denotes the restriction of $a_{2,i}$ to the appropriate variables.) Since $p(x,a_2)$ does not contain any formulas $\ind^2$-Kim-dividing over $M$, by compactness and our choice of $(a_{2,i})_{i<\omega}$, we know $\bigcup_{i<\omega} p(x,a_{2,i})$ is consistent. Therefore, $\{\phi(x,y,a_{2,i}^*) : i <\omega\}$ is consistent. But this is a contradiction, since by assumption and the paragraph before the claim $\phi(x,y,a_2^*)$ $\ind^2$-Conant-divides over $M$.
    \end{proof}
    Let $p_1(x,b_1) := \tp(a_1/Mb_1)$. We want to show there is $a \models p_1(x,b_1) \cup \tp(a_2/Mb_2)$ such that $a$ is $\ind^2$-Kim-independent from $b_1b_2$ over $M$. By (iv) and right extension for $\ind^1$, there is $b_1'' \equiv_{Mb_2} b_1$ such that $b_1'' \ind_M^1 b'_1b_2$. It thus suffices to show that $p_1(x,b''_1) \cup \tp(a_2/Mb_2)$ has a realisation $a$ that is $\ind^2$-Kim-independent from $b_1''b_2$ over $M$. 

    Since $b_1'' \equiv_M b_1 \equiv_M b'_1$, by (iv) and left extension for $\ind^1$, we can find some $b_2'$ such that $b''_1b'_2 \equiv_M b_1'b_2$ and $b''_1b'_2 \ind^2_M b_1'b_2$ (since $\ind^1 \implies \ind^2$). By right extension for $\ind^2$, it follows that $b'_1b_2, b_1''b_2'$ begin an $\ind^2$-Morley sequence over $M$. So by \thref{chain-condition}, monotonicity and an automorphism, there is some $a \equiv_{Mb'_1b_2} a_2$ such that $a$ is $\ind^2$-Conant-independent from $b''_1b_2$ over $M$. Hence, by the first paragraph of the proof, $a$ is $\ind^2$-Kim-independent from $b''_1b_2$ over $M$. Since $a \equiv_{Mb_2} a_2$, we have that $a \models \tp(a_2/Mb_2)$, and since $ab''_1 \equiv_M ab'_1 \equiv_M a_2b'_1 \equiv_M a_1b_1$ by the Claim, we have $a \models p_1(x,b''_1)$.
\end{proof}
\begin{remark}
    Since the compatibility condition is clear in this case, note that, by \thref{symmetry-characterisation}, we can slightly weaken the above result by replacing ``$\ind^2$-Kim-independence satisfies GUWP w.r.t. $\ind^2$'' in the statement with (the oftentimes easier-to-check) ``$\ind^2$-Kim-independence satisfies left transitivity''.
\end{remark}
\begin{example}
    We again recover several results from the literature:
    \begin{enumerate}[(i)]
        \item If $T$ is simple, then by \thref{find-symmetry-simple} $\find$ is symmetric and thus satisfies left and right extension. Since $\find$ always satisfies weak transitivity, the above recovers the independence theorem over models in simple theories \cite[Theorem 3.5]{kim1997simple}.
        \item If $T$ is $\NSOP_1$ and we take $\ind^1 = \uind$ and $\ind^2 = \iind$, this recovers \cite[Proposition 6.1]{kaplan2020kim}.
        \item If $T$ is $\NSOP_2$ ($= \NSOP_1$ by \cite{mutchnik2025nsop}) and we take $\ind^1 = \ind^2$ to be canonical coheir independence, this recovers \cite[Proposition 5.7]{mutchnik2025nsop}.
        \item If $T$ is a free amalgamation theory, there is a ``free amalgamation relation'' $\ind$ in the sense of Conant \cite{conant2017axiomatic}. By definition, such a relation always satisfies existence, monotonicity, left and right extension, weak transitivity, and symmetry (among other properties). Moreover, since it is stationary, $\ind$ satisfies GUWP w.r.t. $\ind$. Therefore, the above recovers the version of the independence theorem for free amalgamation theories that appears in the introduction of \cite{mutchnik2025nsop}.
    \end{enumerate}
    Observe that, in (ii) and (iii), we need to assume symmetry of Kim-independence \cite[Theorem 5.16]{kaplan2020kim} and $\uind$-Conant-independence \cite[Theorem 5.4]{mutchnik2025nsop} in their respective contexts, which we cannot (yet) recover abstractly only from \thref{symmetry-characterisation,examples-of-uwp}.
\end{example}
There is a slight modification of the assumptions of the theorem above that provides a sufficient condition for $\NSOP_1$:
\begin{proposition} \thlabel{independence-theorem}
    Let $\ind$ be an independence relation satisfying full existence and monotonicity. Suppose that $\ind$-Kim-independence satisfies left and weak transitivity and GUWP w.r.t. $\ind$. Suppose further that $\ind \implies \ind$-Kim-independence. Then $T$ is $\NSOP_1$ and $\ind$-Kim-independence $=\kind$.
\end{proposition}
\begin{proof}
    First, we claim that, given our assumptions, $\ind$-Kim-independence satisfies the independence theorem over models (as in \thref{properties-of-independence}(xvii)). Note that the first part of the previous proof (including the claim) still goes through. For the second part, simply note that $\ind$-Kim-independence satisfies left (and right) extension. To show this, we can either use symmetry (which follows from \thref{symmetry-characterisation}) and \thref{mon-and-strong-fin-char-are-preserved}, or GUWP and \thref{mon-and-strong-fin-for-universal-kim}. Moreover, since we assumed weak transitivity for $\ind$-Kim-independence, we can still apply \thref{chain-condition}. So the argument works. 

    Now, observe that $\ind$-Kim-independence satisfies strong finite character, existence, and monotonicity by \thref{mon-and-strong-fin-char-are-preserved}, symmetry by \thref{symmetry-characterisation}, and the independence theorem over models by the previous paragraph. Hence, together with GUWP w.r.t $\ind$, \cite[Theorem 6.11]{kaplan2021transitivity} implies that $T$ is $\NSOP_1$ and $\ind$-Kim-independence and $\kind$ coincide.
\end{proof}
With this proposition, we can also devise a new criterion for a certain subclass of $\NSOP_1$ theories that does not explicitly invoke any witnessing properties. For this, we need a few intermediate lemmas that use weak transitivity.
\begin{lemma}
    Let $\ind$ be an independence relation satisfying full existence, monotonicity, and weak transitivity. If $\phi(x,b)$ $\ind$-Kim-divides over $M$, then it $\ind^{\textnormal{opp}}$-Kim-divides over $M$. 
    
    In particular, $\ind^{\textnormal{opp}}$-Kim-independence implies $\ind$-Kim-independence.
\end{lemma}
\begin{proof}
    Suppose there is $\phi(x,b)$ and an $\ind$-Morley sequence $(b_i)_{i < \omega}$ over $M$ with $b_0 = b$ such that $\{\phi(x,b_i) : i < \omega\}$ is inconsistent, and so $k$-inconsistent for some $k$ by $M$-indiscernibility. By \cite[Claim 7]{delbee2023axiomatic}, there exists some $M$-indiscernible sequence $(b'_i)_{i < \omega}$ such that $b'_0 \dots b'_n \equiv_M b_n \dots b_0$ for all $n < \omega$. In particular, the set $\{\phi(x, b'_i) : i < \omega\}$ is also $k$-inconsistent. 

    It remains to show that $(b'_i)_{i < \omega}$ is $\ind^{\textnormal{opp}}$-independent over $M$. For any $n < \omega$, we have $b'_0 \ind_M b'_1 \dots b'_n$, so by $M$-indiscernibility $b'_1 \ind_M b'_2 \dots b'_{n+1}$ and thus by right monotonicity $b'_1 \ind_M b'_2 \dots b'_n$. Hence, by weak transitivity, $b'_0b'_1 \ind_M b'_2 \dots b'_n$. Proceeding inductively, we obtain that $b'_0\dots b'_{n-1} \ind_M b'_n$, as required.
\end{proof}
\begin{example}
    If a formula divides along a $\uind$-Morley sequence over a model $M \models T$, then it also divides along an $\hind$-Morley sequence over $M$.
\end{example}
\begin{lemma}\thlabel{guwp-and-opp-weight}
    Let $\ind^1$ and $\ind^2$ be independence relations satisfying full existence and monotonicity. Assume that $\ind^1$ also satisfies weak transitivity. Suppose that, for every $M \models T$, every $(\ind^1)^{\textnormal{opp}}$-independent sequence $(a_i)_{i<\size{T}^+}$ over $M$, and any finite tuple $b$, there is some $i < \size{T}^+$ such that $\tp(b/Ma_i)$ does not $\ind^2$-Kim-divide over $M$. Then $\ind^1$ satisfies GUWP w.r.t. $\ind^2$.
\end{lemma}
\begin{proof}
    Same proof as in \thref{guwp-and-weight}((ii) $\Rightarrow$ (i)), noting that we may extract and extend an $(\ind^1)^{\text{opp}}$-Morley sequence using the claim from the previous proof.
\end{proof}
\begin{proposition} \thlabel{local-character-and-guwp}
    Let $\ind$ be an independence relation satisfying full existence and monotonicity. Suppose that $\ind$-Kim-independence is compatible with $\ind$ and satisfies weak and left transitivity, and left extension. The following are equivalent:
    \begin{enumerate}[(i)]
        \item $\ind$-Kim-independence satisfies GUWP w.r.t. $\ind$. 
        \item $\ind$-Kim-independence satisfies chain local character. 
    \end{enumerate}
\end{proposition}
\begin{proof}
    ((i) $\Rightarrow$ (ii)) This follows by left transitivity, \thref{symmetry-characterisation,symmetry-implies-local-character}. 

    ((ii) $\Rightarrow$ (i)) We use \thref{guwp-and-opp-weight}. So let $M \models T$ be a model and $b$ be a finite tuple. Let $(a_i)_{i < \size{T}^+}$ be such that $a_{<i}$ is $\ind$-Kim-independent from $a_i$ over $M$. Observe that $M$ is $\ind$-Kim-independent from $a$ over $M$. Thus, by left extension for $\ind$-Kim-independence, we can find a continuous chain of models $(M_i)_{i < \size{T}^+}$ such that $M \subseteq M_i$, $a_{<i} \subset M_i$, and $M_i$ is $\ind$-Kim-independent from $a_i$ over $M$ for all $i < \size{T}^+$. 

    By chain local character, we can find some $i < \size{T}^+$ such that $b$ is $\ind$-Kim-independent from $M_{\size{T}^+} := \bigcup_{j < \size{T}^+} M_j$ over $M_i$. Thus, by right monotonicity, $b$ is $\ind$-Kim-independent from $a_i$ over $M_i$. Therefore, by left transitivity, it follows that $b$ is $\ind$-Kim-independent from $a_i$ over $M$. 
\end{proof}
\begin{corollary}\thlabel{subclass-of-nsop1-characterisation}
    Let $\ind$ be an independence relation satisfying full existence and monotonicity. Suppose that $\ind$-Kim-independence satisfies weak and left transitivity, left extension, and chain local character. Suppose further that $\ind \implies \ind$-Kim-independence. Then $T$ is $\NSOP_1$ and $\ind$-Kim-independence $= \kind$.
\end{corollary}
\begin{proof}
    Combine \thref{independence-theorem,local-character-and-guwp}.
\end{proof}
This imposes some restrictions on any potential converse to \thref{symmetry-implies-local-character}.

Returning to the $\ind'$-independence theorem over models, we can adapt the argument in \cite[Proposition 6.3]{dobrowolski2022independence} to obtain a partial converse to \thref{wit}:
\begin{theorem} \thlabel{wit-and-guwp}
    Let $\ind^1 \implies \ind^2$ be independence relations satisfying full existence and monotonicity such that $\ind^1$ also satisfies left and right extension. If $\ind^2$-Conant-independence satisfies the $\ind^2$-independence theorem over models, then $\ind^1$ satisfies GUWP w.r.t. $\ind^2$. 
\end{theorem}
\begin{proof}
    Assume, for contradiction, that $\ind^1$ does not satisfy GUWP w.r.t. $\ind^2$. So there is a formula $\phi(x,a)$, an $\ind^1$-Morley sequence $(a_i)_{i < \omega}$ over $M$ and an $\ind^2$-Morley sequence $(b_i)_{i < \omega}$ over $M$ such that $a_0 = b_0$, $\{\phi(x,a_i) : i < \omega\}$ is consistent, and $\{\phi(x,b_i) : i < \omega\}$ is inconsistent. Pick $c \models \{\phi(x,a_i) : i < \omega\}$ such that $(a_i)_{i<\omega}$ is $Mc$-indiscernible, which implies using $a_0 = b_0$ and \thref{univ-kim-forking-iff-univ-kim-dividing,universal-kim-preserves-implications} that $c$ is $\ind^2$-Conant-independent from $b_0$ over $M$. 

    By indiscernibility, there is some $N< \omega$ such that $\{\phi(x, b_i) : i \leq N\}$ is inconsistent. We will find by induction on $k \leq N$ some $c_k \models \{\phi(x,b_i) : 0 \leq i \leq k\}$ such that $c_i$ is $\ind^2$-Conant-independent from $b_{\leq k}$ over $M$, which will give us the required contradiction. 
    
    For the base case, we can take $c_0 = c$. For the inductive step, suppose that we have found $c_k \models \{\phi(x, b_i) : i \leq k\}$ such that $c_k$ is $\ind^2$-Conant-independent from $b_{\leq k}$ over $M$. Since $b_{k+1} \equiv_M b_k$, we can pick $\sigma \in \Aut(\M/M)$ such that $\sigma(b_k) = b_{k+1}$ and let $c' := \sigma(c)$, so that $c' \equiv_M c_k$. Hence, by invariance, $c'$ is $\ind^2$-Conant-independent from $b_{k+1}$ over $M$. Moreover, by choice, $b_{k+1} \ind^2_M b_{\leq k}$, and by induction hypothesis, $c_k$ is $\ind^2$-Conant-independent from $b_{\leq k}$ over $M$. Therefore, by the $\ind^2$-independence theorem over models, there exists some $c_*$ such that $c_* \equiv_{Mb_{\leq k}} c_k$, $c_* \equiv_{Mb_{k+1}} c'$, and $c_*$ is $\ind^2$-Conant-independent from $b_{\leq k+1}$ over $M$. Moreover, $c_* \models \{\phi(x,b_i) : i \leq k+1\}$. So we let $c_{k+1} := c_*$. This completes the induction. 
\end{proof}
With this lemma, we can obtain an analogue to \cite[Proposition 5.2]{chernikov2016model}. The innovation here is that, while Chernikov and Ramsey's proof involves the construction of a tree and the syntactic properties of $\SOP_1$, our proof is completely ``semantic'' and only involves properties of abstract independence relations (modulo the equivalence of $\NSOP_1$ and Kim's Lemma for coheir Morley sequences):
\begin{proposition} \thlabel{conant-independence-satisfying-wit-implies-nsop1}
    Let $\uind \implies \ind$ be an independence relation. If $\ind$-Conant-independence satisfies the $\ind$-independence theorem over models, then $T$ is $\NSOP_1$.  
\end{proposition}
\begin{proof}
    By \thref{wit-and-guwp}, $\uind$ satisfies GUWP w.r.t. $\ind$. As $\uind \implies \ind$, by \thref{relative-dominating-between-relations} $\uind$ satisfies GUWP w.r.t. $\uind$. Therefore, by \cite[Theorem 3.16]{kaplan2020kim}, $T$ is $\NSOP_1$.
\end{proof}
\begin{corollary} \thlabel{conant-coheir-independence}
    The following are equivalent:
    \begin{enumerate}[(i)]
        \item $T$ is $\NSOP_1$.
        \item Coheir-Conant-independence satisfies the independence theorem over models. 
        \item Conant-independence satisfies the independence theorem over models.
    \end{enumerate}
    In this case, coheir-Conant-independence, Conant-independence, and $\kind$ coincide.
\end{corollary}
\begin{remark}
    Let us note that, in the proof of \thref{wit-and-guwp}, we only need left extension of $\ind^1$ to obtain the equivalence between $\ind^1$-Conant-forking and $\ind^1$-Conant-dividing. Thus, if we redo the proof only working with non-$\ind^2$-Conant-dividing, we can show the equivalence of the following two semantically:
    \begin{enumerate}[(i)]
        \item $\iind$ satisfies GUWP w.r.t. $\iind$. 
        \item Non-Conant-dividing independence satisfies the independence theorem over models. 
    \end{enumerate}
    Of course, once we can appeal to the syntactic features of $\NSOP_1$ theories, this is again equivalent to the conditions from the above results.
\end{remark}
\begin{example}
    We have seen before that, if $T$ is a free amalgamation theory with relation $\ind$, then Conant-independence satisfies the $\ind$-independence theorem over models, and there are strictly $\NSOP_4$ such theories. In particular, as any free amalgamation relation satisfies left extension, Conant-independence fails the independence theorem in any strictly $\NSOP_4$ free amalgamation theory.
\end{example}
\section{Dichotomies}\label{sec:dichotomies}
In this section, we will prove a result that illustrates how we can use the witnessing framework developed in \S\ref{sec:witnessing} to prove some old and new dichotomies across dividing lines. A precedent to this type of result can be found in the proofs in \cite{dmitrieva2023dividing} of the facts that simple theories are either stable or $\IP$, and $\NSOP_1$ theories are either simple or $\TP_2$. Although these results, in their original form, trace back to Shelah, the proofs in \cite{dmitrieva2023dividing} use semantic arguments involving Kim-independence to show this, in contrast to the original, combinatorial proofs. In what follows, we will prove an extension of these dichotomies that covers both of them and has new applications within the context of $\NSOP_4$ theories.

The key lemma is the following technical result:
\begin{lemma} \thlabel{collapse-of-ordering}
    Let $\ind^1$, $\ind^2$ and $\ind^3$ be independence relations satisfying full existence and monotonicity. Suppose that:
    \begin{enumerate}[(i)]
        \item Kim-strict $\ind^2$ satisfies full existence. 
        \item $\ind^2 \implies \ind^3$. 
        \item $\ind^1$ satisfies GUWP w.r.t. $\ind^2$.
        \item Kim-strict $\ind^3$ satisfies GUWP w.r.t $\ind^1$.
    \end{enumerate}
    Then $\leq_1$ is trivial.
\end{lemma}
\begin{proof}
    Assume $\phi(x,b)$ $\ind^1$-Kim-divides over $M$. By (iv), it follows that $\phi(x,b)$ Kim-strict $\ind^3$-Conant-divides over $M$. By (i), $\tp(b/M)$ has a global Kim-strict $M$-$\ind^2$-free extension, which by (ii) is also $M$-$\ind^3$-free. Thus, $\phi(x,b)$ $\ind^2$-Kim-divides over $M$. Hence, by (iii), it follows that $\phi(x,b)$ $\ind^1$-Conant-divides over $M$, as required. 
\end{proof}
\begin{remark}
    It is clear that the use of Kim-independence in the previous proof is not necessary. We could have stated the more general result using another independence relation $\ind^4$ and the corresponding $\ind^4$-strict versions of $\ind^2$ and $\ind^3$. However, we have preferred to avoid this since Kim-independence is all we will use later on, and also to make it less cumbersome.
\end{remark}
As we will see in what follows, the above result has important applications in neostability theory. We need the following definition from \cite[Definition 3.2]{chernikov2012forking}:
\begin{definition}
    Let $A$ be a set and $b$ a tuple. We say $\phi(x,b)$ \textbf{quasi-divides over $A$} if there exist $m < \omega$ and $(b_i)_{i<m}$ with $b_i \equiv_A b$ for all $i < m$ such that the set $\{\phi(x,b_i) : i < m\}$ is inconsistent. 
\end{definition}
It is clear that, in order to apply \thref{collapse-of-ordering}, we need to find instances of independence relations $\ind$ for which Kim-strict $\ind$ has full existence. 
The key is to make the following observation: the proof of \cite[Corollary 3.16]{kim2022some} does not really use any specific features of $\NATP$ theories; rather, it just requires the assumption (which is now a theorem; see \cite[Corollary 2.25]{kruckman2024new}) that every formula that Kim-forks over $M$ also quasi-divides over $M$. Thus, this gives us a general theorem:
\begin{lemma} \thlabel{existence-of-kim-strict-free-ext}
    Let $\ind$ be an independence relation satisfying full existence, monotonicity, and quasi-strong finite character. Let $M \models T$. Then every type over $M$ has a global Kim-strict $\ind$-free extension over $M$. In other words, Kim-strict $\ind$ satisfies full existence.
\end{lemma}
Putting all of this together, we obtain the following dichotomy:
\begin{proposition} \thlabel{nsop1-or-btp-dichotomy}
    Let $\ind \implies \iind$ be an independence relation satisfying full existence, monotonicity, and quasi-strong finite character such that $\iind$ satisfies GUWP w.r.t. $\ind$. Then $T$ is either $\NSOP_1$ or $\BTP$.
\end{proposition}
\begin{proof}
    Suppose that $T$ is $\NBTP$. By the New Kim's Lemma (\cite[Theorem 5.2]{kruckman2024new}), Kim-strict $\iind$ satisfies GUWP w.r.t. $\iind$. Moreover, by assumption, $\iind$ satisfies GUWP w.r.t. $\ind$, and quasi-strong finite character and \thref{existence-of-kim-strict-free-ext}, Kim-strict $\ind$ satisfies full existence. Therefore, by \thref{collapse-of-ordering}, it follows that $\leq_{\text{i}}$ is trivial. Hence, $T$ is $\NSOP_1$. 
\end{proof}
Note that, from this, we can use Shelah's dichotomies to show that, in this context, $T$ is either simple or $\TP_2$, and stable or $\IP$. Moreover, note that the quasi-strong finite character condition means that the class of theories where an independence relation as in the hypothesis of \thref{nsop1-or-btp-dichotomy} contains the class of $\NSOP_1$ theories, since we can take $\ind = \iind$. 

The main application of the above result shows that most known $\NSOP_4$ theories from the literature are $\BTP$. To see this, let us note that the independence relations considered by Mutchnik in \cite{mutchnik2024conant} satisfy the assumptions of our theorem:
\begin{lemma} \thlabel{stationarity-implies-qsfc}
    If $\ind$ satisfies full existence, monotonicity, and stationarity over a set $C$, then it also satisfies quasi-strong finite character over $C$.
\end{lemma}
\begin{proof}
    Let $C$ be a set of parameters such that $\ind$ satisfies stationarity and full existence over $C$, and let $p(x),q(y) \in S(C)$. Pick $b_0 \models q$. By full existence, there is $a_0 \models p$ such that $a_0 \ind_C b_0$. Let $\Sigma_{p,q}(x,y) := \tp(a_0b_0/C)$. We claim that this $\Sigma_{p,q}$ works.

    Indeed, if $ab \models \Sigma_{p,q}(x,y)$, then $ab \equiv_C a_0b_0$ by definition, and thus by invariance $a \ind_C b$. Conversely, suppose that $a \models p(x)$ and $b \models q(y)$ are such that $a \ind_C b$. As $b \equiv_C b_0$, there is some $\sigma \in \Aut(\M/C)$ such that $\sigma(b) = b_0$. Let $a' := \sigma(a)$. Then $a'b_0 \equiv_C ab$, so $a' \ind_C b_0$. By assumption, $a_0 \ind_C b_0$. Thus, as $a' \equiv_C a_0$, it follows from stationarity that $a' \equiv_{Cb_0} a_0$. But then $a_0b_0 \equiv_C a'b_0 \equiv_C ab$, as required.
\end{proof}
\begin{corollary}
    \begin{enumerate}[(i)]
        \item Strictly $\NSOP_4$ free amalgamation theories are $\BTP$. 
        \item The model companion of the theory of graphs without cycles of length $\leq n$ is $\BTP$. 
        \item Strictly $\NSOP_4$ Hrushovski constructions with a ``good'' control function are $\BTP$. 
        \item The generic $\mathbf{H}_4$-free 3-hypertournament is $\BTP$.
    \end{enumerate}
\end{corollary}
\begin{proof}
    In all cases, we use \thref{nsop1-or-btp-dichotomy}. For (i)-(iii), we can use the stationary independence relations from \cite{mutchnik2024conant} together with \thref{stationarity-implies-qsfc}. For (iv), we can use $\ind^{\text{hti}}$ from \cite{miguel2024classification}.
\end{proof}
\begin{remark}
    Let us mention in passing that the result for $T_{\mathbf{H}_4\text{-free}}$ can be improved in a similarly ``semantic'' fashion, although in a way that does not seem to directly generalise to cover other cases. The idea is to adapt the notion of reliable extensions introduced by Hanson in \cite{hanson2023bi} to $\ind^{\text{hti}}$-free extensions of types, and use \cite[Proposition 2.6]{hanson2023bi} to conclude that $T_{\mathbf{H}_4\text{-free}}$ must also be $\CTP$.
\end{remark}
\bibliographystyle{amsalpha}
\bibliography{list}
\end{document}